\documentclass[journal]{IEEEtran}
\ifCLASSINFOpdf
\else
\fi
\usepackage{amsmath,amssymb,amsfonts,amsthm,dsfont} 
\usepackage{graphics}
\usepackage[dvips]{graphicx}
\usepackage[table,usenames,dvipsnames]{xcolor}
\usepackage{subcaption}
\usepackage{tikz-cd}
\usepackage{algorithm2e}
\usepackage[breaklinks=true, colorlinks, bookmarks=true, citecolor=Black, urlcolor=Violet,linkcolor=Black]{hyperref}

\usepackage{hyperref}

\newtheorem{theorem}{Theorem}
\newtheorem{lemma}{Lemma}

\newtheorem{assum}{Assumption}
\theoremstyle{definition}

\newcommand{\pa}{\partial}
\newcommand{\ba}{\begin{align}}
\newcommand{\ea}{\end{align}}
\newcommand{\fr}{\frac}
\newcommand{\gam}{\gamma}
\newcommand{\ep}{\varepsilon}

\newcommand{\R}{{\mathbb R}}

\begin{document}
\setlength\abovedisplayskip{8pt}
\setlength\belowdisplayskip{8pt}
\setlength\abovedisplayshortskip{8pt}
\setlength\belowdisplayshortskip{8pt}
 
\allowdisplaybreaks
 
\setlength{\parindent}{1em}
\setlength{\parskip}{0em}
%
\title{Safe PDE Backstepping QP Control with \\High Relative Degree CBFs:\\ Stefan Model with Actuator Dynamics}
%
%
%

\author{Shumon~Koga,~\IEEEmembership{Member,~IEEE,}
        and~Miroslav~Krstic,~\IEEEmembership{Fellow,~IEEE}
\thanks{S. Koga is with the Department
of Electrical and Computer Engineering, University of California at San Diego, La Jolla,
CA, 92093-0411 USA (e-mail: skoga@ucsd.edu).}
\thanks{M. Krstic is with the Department
of Mechanical and Aerospace Engineering, University of California at San Diego, La Jolla,
CA, 92093-0411 USA (e-mail: krstic@ucsd.edu).}}

\maketitle

\begin{abstract}
High-relative-degree control barrier functions (hi-rel-deg CBFs) play a prominent role in automotive safety and in robotics. In this paper we launch a generalization of this concept for PDE control, treating a specific, physically-relevant model of thermal dynamics where the boundary of the PDE moves due to a liquid-solid phase change---the so-called Stefan model. The familiar QP design is employed to ensure safety but with CBFs that are infinite-dimensional (including one control barrier ``functional'') and with safe sets that are infinite-dimensional as well. Since, in the presence of actuator dynamics, at the boundary of the Stefan system, this system's main CBF is of relative degree two, an additional CBF is constructed, by backstepping design, which ensures the positivity of all the CBFs without any additional restrictions on the initial conditions. It is shown that the ``safety filter'' designed in the paper guarantees safety in the presence of an arbitrary operator input. This is similar to an automotive system in which a safety feedback law overrides---but only when necessary---the possibly unsafe steering, acceleration, or braking by a vigorous but inexperienced driver. Simulations have been performed for a process in metal additive manufacturing, which show that the operator's heat-and-cool commands to the Stefan model are being obeyed but without the liquid ever freezing. 
\end{abstract}

\begin{IEEEkeywords}
Control barrier function (CBF), distributed parameter systems, Stefan problem, nonlinear control, safety, quadratic program, additive manufacturing.  
\end{IEEEkeywords}

%
\IEEEpeerreviewmaketitle

\section{Introduction}
\paragraph{Safety, hi-rel-deg CBFs, and non-overshooting control}
{Guaranteed} safety is a necessity  in most engineering applications, including robotic and automotive systems. Several approaches have been developed for ensuring  safety, such as reachability analysis \cite{bansal2017hamilton}, contraction theory \cite{singh2017robust}, model predictive control \cite{hewing2020learning}, and so on.  
Since the seminal work on control barrier functions (CBF)~\cite{ames2016control}, CBF-based designs for safety have garnered tremendous  attention. While they come with Lyapunov-like characterizations, CBFs are actually system outputs, which need to be kept positive. An alternative perspective is that keeping those inputs positive ensures that a certain desired set is positively/forward-invariant for the system. 

The introduction of high relative degree CBFs in 2016 \cite{nguyen2016exponential}, and their further nonlinear refinement \cite{xiao2019control}, constitute breakthroughs in removing the  relative degree one restrictions of the early CBF work. Much progress has followed, including~\cite{JANKOVIC2018359,breeden2021high,XU2018195,dhiman2020control}, to mention a few.

While not called `high-relative-degree CBFs' (or CBFs at all), they do first appear in the second author's 2006 article \cite{krstic2006nonovershooting}, ten years prior to \cite{nguyen2016exponential}. In that article \cite{krstic2006nonovershooting}, for a class of nonlinear systems, the so-called ``non-overshooting control'' problem is solved, as one form of safety used to be referred to in the earlier literature, particularly for setpoint regulation of outputs of linear plants without exceeding the setpoint. 

\paragraph{Safe control for PDEs}

Control with safety guarantees  plays a prominent role in robotics and in collision avoidance for autonomous vehicles. For PDEs, i.e., in infinite dimension, safe control has appeared only where the infinite dimensional state needs to maintain a positive value for reasons of physical validity of the model, such as when the state is distributed concentration \cite{refId0} or gas density \cite{karafyllis2021global}, or when the liquid level must be kept below a certain value to avoid spilling \cite{karafyllis2021spillfree}.

The present paper makes a significant advance to this inquiry---PDE safety. We consider a thermal system, known as the Stefan PDE-ODE system, which models a solid-liquid phase change and whose liquid needs to be kept from developing islands of solid in its midst, while an operator is pursuing his objective which may, for example, be the relocation of the liquid-solid interface to a desired position. 

While we did achieve in \cite{Shumon19journal} the result of regulating the interface to a setpoint while maintaining the entire liquid in that phase, we had done that using direct actuation of the heat flux at the liquid boundary. Realistic actuation does not have direct access to heat flux but is performed with an electrical actuator whose input is voltage. In other words, the input for which a controller was designed in \cite{Shumon19journal} and the realistic voltage input are separated by at least an integrator. This addition of an integrator changes the meaning of safety for the Stefan model. Even though just a single state variable has been added to a model of infinite dimension, the dimension did increase by one and with that the geometry of the safe set that needs to be maintained forward invariant changes. An alternative perspective is that, with the addition of an integrator, the relative degrees of the CBFs for the Stefan system have increased by one. 

Hence, a non-trivial modification to the feedback in \cite{Shumon19journal} is needed in order to retain what has been achieved there but in the presence of an integrator at the input. This retention of safety is accomplished by suitably employing, and extending, the backstepping design idea for safety from \cite{krstic2006nonovershooting}. 

\paragraph{QP modifications for  Stefan PDE and  actuator ODE} \label{sec:Intro-QP} 

The current paper differs from \cite{refId0,karafyllis2021global,karafyllis2021spillfree} not only in the sense of the physics considered---thermal here and bio-populations, gas dynamics, and free-surface flows in those papers. Another difference is that what was achieved safely in those papers  is stabilization, whereas in this paper we tackle a more general objective of safeguarding the system from the potentially unsafe input being applied by an external operator---similar to safeguarding a vehicle, using dynamic stability control, from the unsafe, overly aggressive actions of an inexperienced driver who doesn't have a good feel for his car's dynamics. We let the operator manipulate the input voltage to the liquid-solid system for as long as his inputs won't create solid islands in the liquid. But for inputs that would violate the phase change constraints, we override the operator with a feedback law that guarantees safety. 

This safety override is performed by a QP-feedback modification of the nominal input. Since this QP modification relies on a backstepping transformation to introduce a CBF of relative degree one to which QP can be applied, we call this a ``QP-backstepping-CBF'' design. 

In the Stefan model with actuator dynamics we have two CBFs: one of relative degree one representing the heat flux at the boundary, which must be maintained positive so that the temperature of the liquid remains above freezing and the conditions of the ``maximum principle'' for the heat PDE are met, and the second CBF which is of relative degree two and incorporates the position of the liquid-solid interface and the deviation of the actual temperature from the freezing temperature. A single input must keep both of these CBFs positive. This is achieved by two QP modifications, which, when combined, give a physically natural feedback which saturates the operator's command between two safety feedback laws, one ensuring that the system is heated enough and the other ensuring that it is not overheated. 

\paragraph{Stefan model of phase change and its control}

The Stefan model of the liquid-solid phase change has been widely utilized in various kinds of science and engineering processes, including sea-ice melting and freezing~\cite{koga2019arctic}, continuous casting of steel \cite{petrus12}, cancer treatment by cryosurgeries \cite{Rabin1998}, additive manufacturing for materials of both polymer \cite{koga2019polymer} and metal \cite{koga20laser}, crystal growth~\cite{ecklebe2021toward}, and thermal energy storage systems \cite{koga20experiment}. 
Apart from the thermodynamical model, the Stefan PDE-ODE systems have been employed to model several chemical, electrical, social, and financial dynamics such as lithium-ion batteries \cite{koga2017battery}, tumor growth process \cite{Friedman1999}, neuron growth \cite{demir2021neuron}, domain walls in ferroelectric thin films \cite{mcgilly2015}, spreading of invasive species in ecology \cite{Du2010speading}, information diffusion on social networks \cite{Wang20book}, and the American put option \cite{Chen2008}.

To our knowledge, efforts on control of the Stefan problem on the full PDE-ODE model commence with the motion planning results in \cite{Hill1967ParabolicEI,COCV_2003__9__275_0}.  
 Approaches employing finite-dimensional approximations are \cite{daraoui2010model,armaou2001robust}.
 For control objectives, infinite-dimensional approaches have been used for stabilization of  the temperature profile and the moving interface of a 1D Stefan problem, such as enthalpy-based feedback~\cite{petrus12} and geometric control~\cite{maidi2014}.  These works designed control laws ensuring the asymptotical stability of the closed-loop system in the ${L}_2$ norm. However, the results in \cite{maidi2014} are established based on the assumptions on the liquid temperature being greater than the melting temperature, which must be ensured by showing the positivity of the boundary heat input. 
 
Recently, boundary feedback controllers for the Stefan problem have been designed via a ``backstepping transformation" \cite{PDEbook,Delaybook,Smyshlyaev2004} which has been used for many other classes of infinite-dimensional systems. For instance, \cite{Shumon19journal} designed a state feedback control law, an observer design, and the associated output feedback control law by introducing a nonlinear backstepping transformation for moving boundary PDE, which achieved  exponential stabilization of the closed-loop system without imposing any {\em a priori} assumption, with ensuring the robustness with respect to the parameters' uncertainty. \cite{koga_2019delay} developed a delay-compensated control for the Stefan problem with proving the robustness to the delay mismatch. Numerous other results can be found in \cite{KKbook2021,koga2021control}.   

\paragraph{Results and contributions of the paper} 
This paper develops two safe control designs for the Stefan PDE-ODE system with actuator dynamics by employing CBF technique. As remarked in Section \ref{sec:Intro-QP}, three CBFs are considered, two of which are given by the physical restrictions in the Stefan model, and the other one is designed to deal with a CBF with relative degree two. First, we develop a non-overshooting control to regulate the liquid-solid interface position at a setpoint position, with satisfying the positivity of the all CBFs. Next, we design a QP-Backstepping-CBF of safety filter for a given nominal or operator input, with ensuring that the closed-loop system satisfies the positivity of all CBFs. Then, revisiting the non-overshooting control, the global exponential stability of the closed-loop system is proven via PDE-backstepping and Lyapunov method. This paper extends our conference paper \cite{koga2022ACC} by 
\begin{itemize}
    \item developing the safe control under the additional constraints from above on states via both nonovershooting and QP control, 
    \item deriving the nonovershooting regulation for the Stefan system with a higher-order actuator dynamics, 
    \item ensuring the safety constraints for the two-phase Stefan system which incorporates the dynamics of the solid phase with a disturbance under the nonovershooting control,
    \item and applying the two safe control methods to a process in metal additive manufacturing of a titanium alloy through numerical simulation. 
\end{itemize}

\paragraph{Organization}

This paper is organized as follows. The one-phase Stefan model with actuator dynamics of the first order and its state constraints are provided in Section \ref{sec:model}. The nonovershooting regulation is derived in Section \ref{sec:nonov}, and a QP-backstepping safety filter is designed in Section \ref{sec:QP}, with providing the theorem. The stability analysis under the nonovershooting regulation is given in Section \ref{sec:stability}. The safe control under the additional constraints from above on states is developed in Section \ref{sec:above}, and the nonovershooting regulation for a higher order actuator dynamics is shown in Section \ref{sec:high}. The two-phase Stefan problem with unknown disturbance is presented in Section \ref{sec:twophase}. The application to metal additive manufacturing is presented in Section \ref{sec:AM}. The paper ends with the concluding remarks in Section \ref{sec:conclusion}.  

\paragraph{Notation and definitions}
Throughout this paper, partial derivatives and $L_2$-norm are denoted as $u_{t}(x,t) = \frac{\partial u}{\partial t} (x,t)$, $u_{x}(x,t) = \frac{\partial u}{\partial x} (x,t)$, and $||u [t]|| = \sqrt{\int_0^{s(t)} u(x,t)^2 dx}$, where $u[t]$ is a function defined on $[0, s(t)]$ with real values defined by $(u[t])(x) = u(x,t) $ for all $x \in [0, s(t)]$. $\R_+ := [0, +\infty)$. $C^0(U;\Omega)$ is the class of continuous mappings on $U \subseteq \R^n$, which takes values in $\Omega \subseteq \R$, and $C^k(U;\Omega)$, where $k \geq 1$ is the class of continuous functions on $U$, which have continuous derivatives of order $k$ on $U$ and takes values in $\Omega$.


\section{Stefan Model and Constraints} \label{sec:model} 

\begin{figure}[t]
\centering
\includegraphics[width=0.99\linewidth]{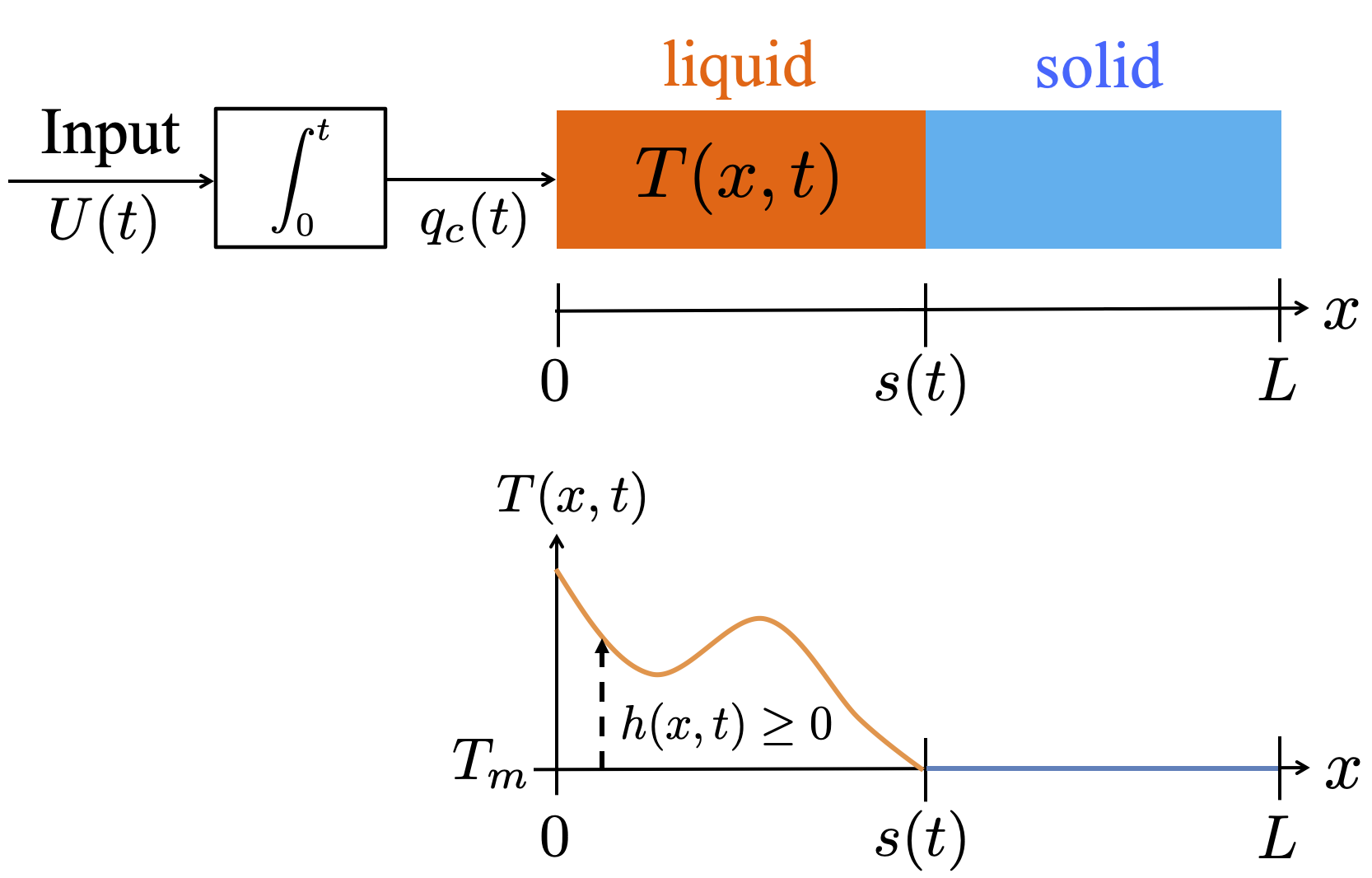}\\
\caption{Schematic of the one-phase Stefan problem with actuator dynamics.}
\label{fig:stefan}
\end{figure}

Consider the melting or solidification  in a material of length $L$ in one dimension (see  Fig.~\ref{fig:stefan}). Divide the domain $[0, L]$ into two time-varying sub-intervals: $[0,s(t)]$, which contains the liquid phase, and $[s(t),L]$, that contains the solid  phase. Let heat flux enter at boundary $x=0$, which affects the distal liquid-solid interface dynamics through heat propagation across liquid phase. Evidently, the heat equation alone does not completely describe the phase transition and must be coupled with the dynamics of the moving boundary. 

The energy conservation and heat conduction laws yield the heat equation of  the liquid phase, the boundary conditions, and the initial values as follows:
\begin{align}\label{eq:stefanPDE}
T_t(x,t)&=\alpha T_{xx}(x,t), \hspace{2mm}  \textrm{for} \hspace{2mm} t > 0, \hspace{2mm} 0< x< s(t), \\ 
\label{eq:stefancontrol}
-k T_x(0,t)&=q_{\rm c}(t),  \hspace{2mm} \textrm{for} \hspace{2mm} t >0,\\ \label{eq:stefanBC}
T(s(t),t)&=T_{{\rm m}}, \hspace{2mm} \textrm{for} \hspace{2mm} t >0, \\
\label{eq:stefanIC}
 s(0) &=  s_0, \textrm{and } T(x,0) = T_0(x),  \textrm{for } x \in (0, s_0],
\end{align}
where $ \alpha :=\fr{k}{\rho C_p}$, and $T(x,t)$,  ${q}_c(t)$,  $\rho$, $C_p$, and $k$ are the distributed temperature of the liquid phase, the boundary heat flux, the liquid density, the liquid heat capacity, and the liquid heat conductivity, respectively. 

In this paper, we model the heater, which produces the heat flux $q_{\rm c}(t)$, as actuated by a voltage input $U(t)$: 
\begin{align} \label{eq:stefan-actuator} 
    \dot q_{c}(t) =  U(t).  
\end{align}
The local energy balance at the liquid-solid interface $x=s(t)$ is 
\begin{align}\label{eq:stefanODE}
\dot s(t)=- \beta T_x(s(t),t) , 
\end{align}
where $\beta:= \fr{k}{\rho \Delta H^*}$, and $\Delta H^*$ represents the latent heat of fusion. In \eqref{eq:stefanODE}, the left hand side represents the latent heat, and the right hand side represent the heat flux by the liquid phase. 
As the moving interface  $s(t)$ depends on the temperature, the problem defined in  \eqref{eq:stefanPDE}--\eqref{eq:stefanODE}  is nonlinear.
The temperature in the solid  is assumed at melting. 

There are two requirements for the validity of the model \eqref{eq:stefanPDE}-\eqref{eq:stefanODE}:
\begin{align}\label{temp-valid}
T(x,t) \geq& T_{{\rm m}}, \quad  \forall x\in(0,s(t)), \quad \forall t>0, \\
\label{int-valid}0 < s(t)<  &L, \quad \forall t>0. 
\end{align}
First, the trivial: the liquid phase is not frozen, i.e., the liquid temperature $T(x,t)$ is greater than the melting temperature $T_{\rm m}$. Second, equally trivially, the material is not entirely in one phase, i.e., the interface remains inside the material's domain. These physical conditions are also required for the existence and uniqueness of  solutions \cite{alexiades2018mathematical}.
Hence, we assume the following for the initial data. 

\begin{assum}\label{ass:initial} 
$0 < s_0 < L$, $T_0(x) \in C^1([0, s_0];[T_{\rm m}, +\infty))$ with $T_0(s_0) = T_{\rm m}$.
 \end{assum}

\begin{lemma}\label{lem1}
With Assumption \ref{ass:initial}, if $q_{\rm c}(t)$ is a bounded piecewise continuous non-negative heat function, i.e.,
\begin{align}
q_{{\rm c}}(t) \geq 0,  \quad \forall t\geq 0,  
\end{align}
then there exists a unique classical solution for the Stefan problem \eqref{eq:stefanPDE}--\eqref{eq:stefanODE}, which satisfies \eqref{temp-valid}, and 
\begin{align} \label{eq:sdot-pos} 
    \dot s(t) \geq 0, \quad \forall t \geq 0. 
\end{align}

\end{lemma}

The definition of the classical solution of the Stefan problem is given in 
Appendix A of \cite{Shumon19journal}. The proof of Lemma \ref{lem1} is by maximum principle for parabolic PDEs and Hopf's lemma, as shown in \cite{Gupta03}.






\section{Nonovershooting Regulation by Backstepping for Multiple CBFs} \label{sec:nonov} 

The regulation of the interface $s(t)$ to a desired setpoint  $s_{\rm r}$ is a crucial task in several applications that involve thermal phase change, such as creating a layer of desired thickness in metal additive manufacturing \cite{koga20laser}. However,  the actuator dynamics given in \eqref{eq:stefan-actuator} introduce a major extra challenge to achieving  setpoint regulation while guaranteeing the safety constraints \eqref{temp-valid}, \eqref{int-valid}. 



In this section we regulate $s(t)$ to $s_{\rm r}$, as well as regulating $T(x,t)$ to $T_{\rm m}$ and $q_{\rm c}$ to zero. This is an equilibrium at the boundary of the safe set. 
Such a mix of stabilization and safety control problems is called ``non-overshooting control'' \cite{krstic2006nonovershooting}. In addition to the two physically imposed CBFs, with backstepping we design a third CBF to ensure safety but without having to  additionally restrict the initial conditions, which is common in other CBF designs. Such an addition of a CBF, ten years after  \cite{krstic2006nonovershooting}, was independently discovered in the format of hi-rel-deg CBFs in \cite{nguyen2016exponential}.

In the next section, we introduce a ``QP-Backstepping-CBF design,'' to allow a safe application of a `nominal feedback,' possibly distinct from the setpoint regulating feedback, or an application of an external operator open-loop input, using a QP selection between the nominal input and backstepping-designed safeguards. 

Let $h_1(t)$, $h_2(t)$, $h_3(t)$, and $h(x,t)$ be CBFs defined by 
\begin{align}
    h_1(t) :&= \sigma(t)  \nonumber\\
    & =  -\left( \fr{k}{\alpha} \int_0^{s(t)} (T(x,t) - T_m) dx + \frac{k}{\beta} (s(t) - s_r) \right), \label{eq:h1_def} \\
    h_2(t) &= q_{\rm c}(t), \label{eq:h2_def}\\
    h_3(t) &= - q_{\rm c}(t) + c_1 \sigma(t), \label{eq:h3_def}\\
    h(x,t) &= T(x,t) - T_{\rm m}. \label{eq:h_def}
\end{align}
The CBF $h_1$ in \eqref{eq:h1_def} represents the ``energy deficit'' (positive, thermal plus potential) relative to the setpoint equilibrium. The added CBF \eqref{eq:h3_def}, seemingly redundant because $h_3=-h_2+c_1 h_1$, represents a backstepping transformation and is crucial for ensuring that $h_1$ is maintained positive. 

\begin{lemma} \label{lem:overshoot} 
    With Assumption \ref{ass:initial}, suppose that the following conditions hold: \begin{align}
     \label{eq:h1-pos}   h_1(t) \geq 0, \\
     \label{eq:h2-pos}     h_2(t) \geq 0,  
    \end{align}
    for all $t \geq 0$. Then, it holds that 
    \begin{align} \label{eq:hxt-pos} 
    h(x,t) \geq & 0, \quad \forall x \in (0, s(t)), \quad \forall t \geq 0, \\
      0< s_0 \leq   s(t) \leq & s_r, \quad \forall t \geq 0.  \label{eq:s-ineq} 
    \end{align}
    \end{lemma}

Lemma \ref{lem:overshoot} is proven with Lemma \ref{lem1}. To validate the conditions \eqref{eq:h1-pos} and \eqref{eq:h2-pos} for all $t \geq 0$, at least the conditions must hold at $t = 0$, which necessitate the following assumptions on the initial condition and the setpoint restriction. 

\begin{assum} \label{ass:h2} 
$0 \leq q_{\rm c}(0) $. 
\end{assum}

\begin{assum}\label{ass:h1} 
The setpoint position $s_{\rm r}$ is chosen to satisfy 
\begin{align}
s_{0} +    \fr{\beta}{\alpha} \int_0^{s_0} (T_0(x) - T_{\rm m}) dx  \leq s_{\rm r} < L. 
\end{align}
\end{assum}

Under these assumptions, we perform a design of a regulating control $U(t)$ so that the conditions \eqref{eq:h1-pos}, \eqref{eq:h2-pos} hold. Taking the first and second time derivatives of \eqref{eq:h1_def}, we have 
\begin{align}
    \dot h_1 &= -h_2 = - q_c, \\
- \ddot h_1 = \dot h_2 &= U. 
\end{align}
Thus, ``energy deficit'' CBF defined by \eqref{eq:h1_def} has relative degree 2, from the voltage input, according to the perspective in  \cite{nguyen2016exponential}, which inherits the backstepping change of variable \eqref{eq:h3_def} from \cite{krstic2006nonovershooting}.

The double integrator  $U\mapsto h_1$ is not unlike a model of adaptive/distance cruise control (ACC), studied in many papers on CBFs and QP. The single integrator  $U\mapsto h_2$ is not unlike a classical (velocity) cruise control problem. Hence, in the Stefan model, achieving safety amounts to simultaneously maintaining safety in terms of both distance and velocity cruise control, without imposing additional restrictions on the initial conditions of the position-like and velocity-like states through design. In addition, in the Stefan model, on top of the two CBFs of relative degrees one and two, one has to maintain the positivity of the CBF  \eqref{eq:h_def}, i.e., to ensure \eqref{eq:hxt-pos}, where the temperature field plays the role of zero dynamics of infinite dimension. 

We design a "non-overshooting control" \cite{krstic2006nonovershooting}, denoted as $U=U^*(\sigma,q_{\rm c})$, so that the following relationships hold: 
\begin{align} 
\label{eq:ddth1}    \dot h_1 &= -c_1 h_1 + h_3 , \\
   \label{eq:ddth2} \dot h_2 &= - c_1 h_2 + c_2 h_3, \\
    \label{eq:ddth3} \dot h_3 &= - c_2 h_3. 
\end{align}
Indeed, $h_3$ in \eqref{eq:h3_def} is defined so that \eqref{eq:ddth1} holds, hence, it suffices to design the control so \eqref{eq:ddth3} holds. 
Taking the time derivatives of \eqref{eq:h3_def}, one 
gets 
\begin{align} \label{eq:nonover} 
    U^*(\sigma,q_{\rm c}) =  - (c_1 + c_2) q_{\rm c} + c_1 c_2 \sigma  . 
\end{align}
With \eqref{eq:nonover}, we  see that \eqref{eq:ddth2} also holds. 

The solution to the linear differential equations \eqref{eq:ddth1}--\eqref{eq:ddth3} is analytically obtained by  
\begin{align} \label{eq:h3-sol} 
    h_3(t) = & h_3(0) e^{- c_2 t}, \\
    h_1(t) 
    = &  h_1(0) e^{-c_1 t} +  \frac{h_3(0) }{c_2 - c_1} ( e^{-c_1 t}  - e^{  - c_2 t} ) , \\
    h_2(t) = & h_2(0) e^{-c_1 t} +  \frac{c_2 h_3(0)  }{c_2 - c_1} ( e^{-c_1 t}  - e^{  - c_2 t} ) . \label{eq:h2-sol}
\end{align}
With Assumptions \ref{ass:h2} and \ref{ass:h1},  $h_1(0) \geq 0$ and $h_2(0) \geq 0$ hold. Thus, if $h_3(0) \geq 0$, one can see that the solutions \eqref{eq:h3-sol}--\eqref{eq:h2-sol} are all nonnegative. For $h_3(0) \geq 0$ to hold, we choose the control gain $c_1$ as
\begin{align} \label{eq:c1_condition} 
     c_1 \geq \fr{q_{\rm c}(0)}{\sigma(0)}
\end{align}
and obtain
the following lemma.

\begin{lemma} \label{lem:safety-up} 
    Let Assumptions \ref{ass:initial}--\ref{ass:h1} hold. Then, the closed-loop system consisting of the plant \eqref{eq:stefanPDE}--\eqref{eq:stefanODE} with the non-overshooting control law \eqref{eq:nonover}, where the gain satisfies \eqref{eq:c1_condition}, guarantees the following to hold:  
    \begin{align}
    h_1(t) \geq 0, \quad h_2(t) \geq 0, \quad h_3(t) \geq 0, \quad \forall t \geq 0.
\end{align}
Moreover, \eqref{eq:hxt-pos} and \eqref{eq:s-ineq} hold. 
\end{lemma}
We  show  stability  in Section \ref{sec:stability}. 


\section{QP-Backstepping-CBF Design of Safety Filter}\label{sec:QP}  

Here, we derive a safety filter to satisfy all CBF constraints for a given operator input. First, we aim to satisfy $h_3 \geq 0$, since it also ensures $h_1 \geq 0$ by the relation of \eqref{eq:ddth1}, as a designed exponential CBF. Namely, we satisfy
\begin{align} \label{eq:ddth3-ineq} 
    \dot h_3 \geq - c_2 h_3. 
\end{align}
Taking the time derivative of \eqref{eq:h3_def} and applying \eqref{eq:ddth3-ineq} leads to the condition of the input as 
\begin{align} \label{eq:Ut-low-ineq} 
    U \leq U^*(\sigma,q_{\rm c}),
\end{align}
where $U^*$ is given in \eqref{eq:nonover}. It remains to ensure $h_2 \geq 0$. Since both $h_1 \geq 0$ and $h_3 \geq 0$ are satisfied under the input condition \eqref{eq:Ut-low-ineq}, and taking into account the fact that $h_3$ can be written with respect to $h_1$ and $h_2$ as defined in \eqref{eq:h3_def}, we aim to satisfy the following inequality: 
\begin{align} \label{eq:ddth2-ineq} 
    \dot h_2 \geq - \bar c_1 h_2 + \bar c_2 h_1, 
\end{align}
for some $\bar c_1 >0$ and $\bar c_2 >0$, which ensures $h_2 \geq 0$. Taking the time derivative of \eqref{eq:h2_def}, and with \eqref{eq:stefan-actuator}, to make \eqref{eq:ddth2-ineq} hold, we arrive at the following condition on the input: 
\begin{align} \label{eq:Ut-up-ineq}
    U_*(\sigma,q_{\rm c}) \leq & U , \\
    U_*(\sigma,q_{\rm c}) &=  - \bar c_1 q_{\rm c} + \bar c_2 \sigma \label{eq:U-h2-const}
\end{align}
%
%
Comparing \eqref{eq:nonover} with \eqref{eq:U-h2-const}, for ensuring the feasibility of the two constrains \eqref{eq:Ut-low-ineq} and \eqref{eq:Ut-up-ineq}, it is sufficient to choose the gains in accordance with the following conditions:
\begin{align} \label{eq:c_bar_condition} 
    \bar c_1 \geq  c_1 + c_2, \quad 0 \leq \bar c_2 \leq c_1 c_2.  
\end{align}

By redefining the gain parameters from $(c_1, c_2, \bar c_1, \bar c_2)$ to $(k_1, k_2, \delta_1, \delta_2)$, one can show that \eqref{eq:nonover} and \eqref{eq:U-h2-const} with the condition \eqref{eq:c1_condition} are equivalent to the following formulation:  
\begin{eqnarray} \label{eq:U_star}
U_* &&= -(k_1+\delta_1) q_c + k_2\sigma \\  \label{eq:U^star}
U^* &&= - k_1 q_c + (k_2 + \delta_2) \sigma\,,
\end{eqnarray}
where 
\begin{align}
k_1 \geq q_c(0)/\sigma(0), \quad k_2 > 0, \label{eq:k-gain-cond} \\
\delta_1 \geq 0, \quad \delta_2 \geq 0 . \label{eq:del-gain-cond}
\end{align} 

Finally, a safety filter, for a given nominal or operator input $U_o(t) $, is designed, inspired by~\cite{ames2016control}, by solving the Quadratic Programming (QP)  problem\footnote{Strictly speaking, the constraints \eqref{U_*<U<U^*} should be written, following the Lie derivative conditions on the CBFs $h_2$ and $h_3$, as, respectively, $ u + \bar c_1 h_2 - \bar c_2 h_1 \geq 0$ and $ - u + c_2 h_3 - c_1 h_2\geq 0$. But we believe that \eqref{U_*<U<U^*} poses less of a chance to confuse the reader.}
\begin{align}
    U = \min_{u \in \R} |u - U_o |^2, \\
    \label{U_*<U<U^*}
\textrm{subject to } \quad   U_* \leq  u \leq & U^*. 
\end{align}
Applying the Karush-Kuhn-Tucker (KKT) optimality condition, the explicit solution is
\begin{align} \label{eq:safety-fin} 
    U  = \min\{ \max\{U_*, U_o\}, U^*\} . 
\end{align}
Since $U_*$ and $U^*$ are designed by backstepping, with an addition of a CBF to given CBFs, we call this safety filter a "QP-Backstepping-CBF" design. 

The cumbersome formula \eqref{eq:safety-fin} can  be rewritten with a saturation function on the operator input $U_o$ as 
\begin{equation} \label{eq:safety-fin-sat}
    U = \left\{
    \begin{array}{ll}
    U^*,& U_o>U^*\\
    U_o, & U_* \leq U_o \leq U^* \\
    U_*, & U_o < U_*\,
    \end{array}
    \right.
\end{equation}
The analysis and design above establish the following.
\begin{theorem} \label{thm:safety} 
Let Assumptions \ref{ass:initial}--\ref{ass:h1} hold. Consider the closed-loop system \eqref{eq:stefanPDE}--\eqref{eq:stefanODE} with QP safety control \eqref{eq:safety-fin-sat}, \eqref{eq:U_star}, and \eqref{eq:U^star}, under an arbitrary operator input $U_o(t)$, where the gain parameters are chosen to satisfy \eqref{eq:k-gain-cond} \eqref{eq:del-gain-cond}. Then, all CBFs defined as \eqref{eq:h1_def}--\eqref{eq:h_def} satisfy the constraints $h_1(t) \geq 0$, $h_2(t) \geq 0$ for all $t \geq 0$, and $h(x,t) \geq 0$ for all $x \in (0, s(t))$ and for all $ t \geq 0$.  
\end{theorem} 

Note that, if $\delta_1$ and $\delta_2$ are chosen as zero,
QP safety control \eqref{eq:safety-fin-sat} becomes
\begin{equation} \label{eq:nonover-fix} 
    U = U_* = U^* = -k_1 q_c + k_2 \sigma, 
\end{equation}
which disregards the operator input $U_o$ and instead performs a non-overshooting regulation to $s=s_r, T=T_m, q_c=0$, whose stability is proven in the next section.


\section{Stability of Non-overshooting Regulation} \label{sec:stability} 


\begin{theorem} \label{thm:nonover} 
Let $s(0)$ and $T(x,0)$ satisfy Assumptions \ref{ass:initial}--\ref{ass:h1}. Consider the closed-loop system \eqref{eq:stefanPDE}--\eqref{eq:stefanODE} with the non-overshooting control law \eqref{eq:nonover-fix}, where the gain parameters satisfy \eqref{eq:k-gain-cond}. Then, all CBFs defined as \eqref{eq:h1_def}--\eqref{eq:h_def} satisfy the positivity constraints. This means, in particular, that, for all $q_{\rm c}(0)>0$, the interface $s(t)$ does not exceed $s_{\rm r}$, the temperature $T(x,t)$ does not drop below $T_{\rm m}$ at any position $x$ between $0$ and $s(t)$, and the heat flux $q_{\rm c}(t)$ never takes a negative value. 
Furthermore, the closed-loop system is  exponentially stable at the equilibrium $s=s_{\rm r}, T(x,\cdot) \equiv T_{\rm m}, q_{\rm c} = 0$,  in the sense of the spatial $L_2$-norm, for all initial conditions in the safe set, i.e., globally. In other words, there exist positive constants $M>0$ and $b>0$ such that the following norm estimate holds: 
\begin{align} \label{eq:norm-estimate} 
    \Phi(t) \leq M \Phi(0) e^{-b t},  
\end{align}
where 
\begin{align} \label{eq:Phi-def} 
\Phi(t):= || T[t] - T_{\rm m} ||^2 +(s(t) - s_{\rm r})^2 + q_{\rm c}(t)^2. 
\end{align}

\end{theorem}

The safety is already proven in Lemma \ref{lem:safety}. The remainder of the section proves stability in Theorem \ref{thm:nonover}. 


Let $X(t)$ be reference error variable defined by $X(t):= s(t)-s_{r}$. Then, the system \eqref{eq:stefanPDE}--\eqref{eq:stefanODE} is rewritten with respect to $h(x,t)$ defined in \eqref{eq:h_def}, $h_2(t)$ defined in \eqref{eq:h2_def}, and $X(t)$ as
\begin{align}\label{u-sys1}
h_{t}(x,t) &=\alpha h_{xx}(x,t),\\
\label{u-sys2}h_x(0,t) &= - h_2(t)/k,\\
\dot h_2(t) &= U(t), \\
\label{u-sys3}h(s(t),t) &=0,\\
\label{u-sys4}\dot X(t) &=-\beta h_x(s(t),t). 
\end{align}

\subsection{Backstepping transformation to target system} 

Following Section 3.3. in \cite{koga2021towards}, we introduce the following forward and inverse transformations:  
\begin{align}\label{eq:DBST}
w(x,t)&=h(x,t)-\frac{\beta}{\alpha} \int_{x}^{s(t)} \phi (x-y)h(y,t) {\rm d}y \notag\\
&-\phi(x-s(t)) X(t), \\
\label{kernel}\phi(x) &= c_1 \beta^{-1} x- \ep, \\
\label{inv-trans}
h(x,t)&=w(x,t)-\frac{\beta}{\alpha} \int_{x}^{s(t)} \psi (x-y)w(y,t) {\rm d}y \notag\\
&-\psi(x-s(t)) X(t), \\
\label{inv-gain}
\psi(x) &= e^{ \bar \lambda x } \left( p_1 \sin\left( \omega x \right) + \ep \cos\left( \omega x \right) \right) , 
\end{align}
where $\bar \lambda = \fr{\beta \varepsilon}{2 \alpha}$, $\omega = \sqrt{\fr{4 \alpha c_1 - (\varepsilon\beta)^2 }{4 \alpha^2 } }$, $p_1 =  - \fr{1}{2 \alpha \beta \omega} \left( 2 \alpha c_1 - (\varepsilon \beta )^2 \right) $, and $0<\ep<2 \fr{\sqrt{\alpha c_1}}{\beta}$ is to be chosen later. 
As derived in Section 3.3. in \cite{koga2021towards}, taking the spatial and time derivatives of \eqref{eq:DBST} along the solution of \eqref{u-sys1}-\eqref{u-sys4}, and noting the CBF $h_3(t)$ defined in \eqref{eq:h3_def} satisfying \eqref{eq:ddth3}, one can obtain the following target system: 
\begin{align}\label{tarPDE}
w_t(x,t)&=\alpha w_{xx}(x,t)+ \dot{s}(t) \phi'(x-s(t))X(t) \notag\\
&+\phi(x-s(t)) d(t), \\
\label{tarBC2} w_x(0,t) &= \fr{h_3(t)}{ k} - \frac{\beta}{\alpha}\varepsilon \left[  w(0,t)-\frac{\beta}{\alpha} \int_{0}^{s(t)} \psi (-y)w(y,t) {\rm d}y \right. \notag\\
&\left. -\psi(-s(t)) X(t) \right].  \\
\dot h_3(t) &= - c_2 h_3(t), \label{tar_h3} \\
\label{tarBC1} w(s(t),t) &= \ep X(t), \\
\label{tarODE}\dot X(t)&=-c_1 X(t)-\beta w_x(s(t),t).
\end{align}
The objective of the transformation \eqref{eq:DBST} is to add a stabilizing term $-c_1 X(t)$ in \eqref{tarODE} of the target $(w,X)$-system which is easier to prove the stability than $(u,X)$-system. 

Note that the boundary condition \eqref{tarBC1} and the kernel function \eqref{kernel} are modified from the one in \cite{Shumon19journal}.  The target system derived in \cite{Shumon19journal} requires ${\mathcal H}_1$-norm analysis for stability proof. However, with the actuation dynamics of the boundary heat input, ${\mathcal H}_1$-norm analysis fails to show the stability. The modification of the boundary condition \eqref{tarBC1} enables to prove the stability in ${ L}_2$ norm as shown later. 

\subsection{Lyapunov  analysis} 

Following Lemma 20 in \cite{koga2021towards}, by introducing a Lyapunov function $V(t)$ defined by 
\begin{align}\label{app:lyap}
V(t) = \fr{1}{2\alpha } || w[t]||^2 + \fr{\varepsilon}{2\beta } X(t)^2,
\end{align}
one can see that there exists a positive constant $\ep^*>0$ such that for all $\ep \in (0, \ep^*)$ the following inequality holds: 
\begin{align}\label{app:dotV2}
\dot V(t) \leq &  - b V (t)    + \fr{ 2 s_{\rm r}}{k^2} h_3(t)^2+ a \dot{s}(t) V(t) , 
\end{align}
where $a = \fr{2\beta \ep }{\alpha} \max \left\{1,\fr{\alpha c^2 s_{{\rm r}}}{2\beta^3 \ep^3} \right\}$, $b =\fr{1}{8} \min \left\{ \fr{\alpha}{s_{{\rm r}}^2}, c \right\}$, and the condition $\dot s(t) \geq 0$ ensured in Lemma \ref{lem1} is applied. We further introduce another Lyapunov function of $h_3$, defined by 
\begin{align} \label{eq:Vh-def} 
    V_h(t) = \fr{p}{2} h_3(t)^2,  
\end{align}
with a positive constant $p>0$. 
Taking the time derivative of \eqref{eq:Vh-def} and applying \eqref{tar_h3} yields 
\begin{align} \label{eq:Vh-der} 
    \dot V_h(t) = p h_3 \dot h_3 = - c_2 p h_3(t)^2. 
\end{align}
Let $\bar V$ be the Lyapunov function defined by 
\begin{align} \label{eq:Vbar-def} 
    \bar V = V + V_h.  
\end{align}
Applying \eqref{app:dotV2} and \eqref{eq:Vh-der} with setting $p = \fr{ 4 s_{\rm r}}{c_2 k^2}$, the time derivative of \eqref{eq:Vbar-def} is shown to satisfy 
\begin{align} \label{eq:Vbar-der} 
    \dot {\bar V}(t) \leq - \bar b \bar V(t) + a \dot{s}(t) V(t),  
\end{align}
where $\bar b = \min\{b, 2 s_{\rm r} / k^2\}$. 
As performed in \cite{koga2021towards}, with the condition $0 < s(t) \leq s_{\rm r}$ ensured in Lemma \ref{lem:overshoot}, the differential inequality \eqref{eq:Vbar-der} leads to 
\begin{align}
    \bar V(t) \leq e^{a s_{\rm r}} \bar V(0) e^{- \bar b t}, 
\end{align}
which ensures the exponential stability of the target system \eqref{tarPDE}--\eqref{tarODE} in the spatial $L_2$-norm. Due to the invertibility of the transformation \eqref{eq:DBST} \eqref{inv-trans}, one can show the exponential stability of the closed-loop system, which completes the proof of Theorem \ref{thm:nonover}.

\section{Constraints from Above \hspace{40mm} on Temperature and Heat Flux} \label{sec:above}

\subsection{Upper bound constraints} 
In practical control systems, the capability of the heater is often restricted to a certain range due to the device limitation. The Stefan problem, as a melting process, is particular in this regard. The heat input should not go beyond a given upper bound, while it is feasible to assume that the lower bound is zero, i.e., the heat actuator does not work as a cooler. Furthermore, the liquid temperature must be lower than some value, which could be the maximum temperature limited by the device or environment for ensuring the safe operation, or the boiling temperature to avoid an evaporation which is another phase transition, from liquid to gas. 

We tackle such an overall state-constrained problem, to guarantee the following conditions to hold: 
\begin{align} \label{eq:T_up} 
   T_{\rm m} \leq & T(x,t) \leq T^*, \quad \forall x \in [0, s(t)], \quad \forall t \geq 0, \\
    0 \leq & q_{\rm c}(t) \leq q^*, \quad \forall t \geq 0,   \label{eq:q_up} 
\end{align}
for some $T^* > T_{\rm m}$ and $q^* > 0$. 
First, we state the following lemma for the Stefan problem.  

\begin{lemma} \label{lem:bound} 
If $T_{m} \leq T_0(x) \leq \Delta \bar{T}_{0}(1 - x/s_0) + T_{m}$ for some $\bar{T}_{0} >0$, and if $0 \leq q_c(t) \leq \bar{q}$ holds for some $\bar{q}>0$ and for all $t\geq 0$, then 
\begin{align} 
T_{m} \leq T(x,t) \leq \bar{T}(x,t) := K(s(t) - x) + T_m ,
\end{align} 
$\forall x\in(0,s(t))$, $\forall t\geq 0$, where $K = \max\{ \bar{q}/k, \Delta \bar{T}_{0}/s_0\}$. 
\end{lemma} 
\begin{proof}
Let $v(x,t) : = \bar{T}(x,t) - T(x,t)$. Taking the time and second spatial derivatives yields 
\begin{align} 
v_t = K \dot{s}(t) - T_{t}(x,t) , \quad v_{xx} = -T_{xx}(x,t) . 
\end{align} 
Since $0 \leq q_c(t)$, we have $\dot{s}(t) \geq 0$. Thus, we obtain 
\begin{align} \label{ch24:vt} 
v_t \geq & \alpha v_{xx} , \\
v_x(0,t) \leq & 0 , \quad v(s(t),t) = 0. \label{ch24:vx} 
\end{align} 
Applying the maximum principle to \eqref{ch24:vt}--\eqref{ch24:vx}, we can state that, if $v(x,0) \geq 0$ for all $x \in (0,s_0)$, then $v(x,t) \geq 0$ for all $x \in (0, s(t))$ and all $t \geq 0$, which concludes Lemma \ref{lem:bound}.  
\end{proof}

Inspired by Lemma \ref{lem:bound}, to guarantee that the upper bound constraint  \eqref{eq:T_up} and \eqref{eq:q_up} are met, the following assumptions on the initial conditions are imposed: 

\begin{assum}\label{assmp:q_c}
$q_{\rm c}(0) \leq \bar q_{\rm c}  := \min\{\fr{k}{s_{\rm r}} (T^* - T_{\rm m}), q^*\}$. 
\end{assum}

\begin{assum} 
$T_{m} \leq T_0(x) \leq \Delta \bar{T}_{0}(1 - x/s_0) + T_{m}$, where $\Delta \bar T_0 := \fr{s_0}{s_{\rm r}} (T^* - T_{\rm m})$. \label{ass:T_up} 
\end{assum}


\subsection{Nonovershooting regulation} 
Then, with $h_3(0) \geq 0$, in addition to the positivity of $h_3(t)$, owing to the monotonically decreasing property of \eqref{eq:h3-sol}, it further holds that 
\begin{align}
    0 \leq h_3(t) \leq h_3(0). 
\end{align}
With this condition, one can see from \eqref{eq:ddth2} that 
\begin{align} \label{eq:ddh2-ineq-up}
    0 \leq \dot h_2 + c_1 h_2 \leq c_2 h_3(0). 
\end{align}
Applying the comparison lemma, the differential inequality \eqref{eq:ddh2-ineq-up} leads to the following inequality for the solution
\begin{align}
    h_2(0) e^{- c_1 t} \leq h_2(t) \leq \left( h_2(0) - \fr{c_2 }{c_1} h_3(0)\right) e^{-c_1 t} + \fr{c_2 }{c_1} h_3(0), 
\end{align}
Considering the lower bound of the left 
\begin{align}
    0 \leq h_2(t) \leq \max \left\{h_2(0), \fr{c_2}{c_1} h_3(0) \right\} ,
\end{align}
and rewriting it as 
\begin{align}
    0 \leq h_2(t) \leq \max \left\{h_2(0), \fr{c_2}{c_1}  \left( c_1 h_1(0) - h_2(0)  \right) \right\} ,
\end{align}
the following condition for the control gain
\begin{align}
    c_2 \leq \fr{c_1 \bar q_{\rm c} }{c_1 \sigma(0) - q_{\rm c}(0)} , \label{eq:c2_condition}
\end{align}
where the right-hand side is positive because of \eqref{eq:c1_condition} and where $\bar q_{\rm c} $ is defined in Assumption~\ref{assmp:q_c}, guarantees that 
\begin{align}
    0 \leq h_2(t)  \leq \bar q_{\rm c} .
\end{align}

\begin{theorem} \label{lem:safety} 

Let $s(0), T(x,0),$ and $q_{\rm c}(0)$ satisfy Assumptions \ref{ass:initial}--\ref{ass:T_up}. Consider the closed-loop system \eqref{eq:stefanPDE}--\eqref{eq:stefanODE} with the non-overshooting control law \eqref{eq:nonover}, where the gain parameters satisfy \eqref{eq:c1_condition} and \eqref{eq:c2_condition}. Then, all CBFs defined as \eqref{eq:h1_def}--\eqref{eq:h_def} satisfy the positivity constraints, and $h_2 \leq \bar q_{\rm c}$ also holds. This means, in particular, that, for all $q_{\rm c}(0)>0$, the interface $s(t)$ does not exceed $s_{\rm r}$, the temperature $T(x,t)$ does not drop below $T_{\rm m}$ and does not exceed the upper limit $T^*>T_{\rm m}$ at any position $x$ between $0$ and $s(t)$, and the heat flux $q_{\rm c}(t)$ never takes a negative value and does not exceed the upper limit $q^*$. 
Furthermore, the closed-loop system is  exponentially stable at the equilibrium $s=s_{\rm r}, T(x,\cdot) \equiv T_{\rm m}, q_{\rm c} = 0$,  in the sense of the spatial $L_2$-norm, for all initial conditions in the safe set, i.e., globally. In other words, there exist positive constants $M>0$ and $b>0$ such that the norm estimate \eqref{eq:norm-estimate} holds with the norm \eqref{eq:Phi-def}.
\end{theorem}

The safety is already proven in this section. The stability proof is identical to the derivation in Section \ref{sec:stability}. 

\subsection{QP-Backstepping-CBF Design} 

We also design QP safety filter for ensuring the upper bounds. In addition to the QP constraint on input in Section \ref{sec:QP}, to ensure the upper bound of $h_2$, we introduce another CBF
\begin{align} \label{eq:h2st-def}
     h^*_2(t) = \bar q_{\rm c}  - h_2(t) . 
\end{align}
We set out to design QP to further satisfy 
\begin{align} 
    \dot h^*_2 \geq - c^*_2 h^*_2, 
\end{align}
for some $c^*_2>0$, which leads to the condition
\begin{align}
    U \leq c_2^* (\bar q_{\rm c}  - h_2(t)) . 
\end{align}
Thus, by setting $c_2^* = k_1$, we reformulate the upper bound of the control in \eqref{eq:U^star} as 
\begin{align} \label{eq:U^star-up}
    U^* = - k_1 q_c + \max \{(k_2 + \delta_2 ) \sigma, k_1 \bar q_{\rm c} \}, 
\end{align}
We state the following Theorem. 
\begin{theorem} \label{thm:safety-up} 
Let Assumptions \ref{ass:initial}--\ref{ass:T_up} hold. Consider the closed-loop system \eqref{eq:stefanPDE}--\eqref{eq:stefanODE} with QP safety control \eqref{eq:safety-fin-sat}, \eqref{eq:U_star}, and \eqref{eq:U^star-up}, under an arbitrary operator input $U_o(t)$, where the gain parameters are chosen to satisfy \eqref{eq:k-gain-cond} \eqref{eq:del-gain-cond}. Then, all CBFs defined as \eqref{eq:h1_def}--\eqref{eq:h_def} and \eqref{eq:h2st-def} satisfy the constraints $h_1(t) \geq 0$, $h_2(t) \geq 0$, $h^*_2(t) \geq 0$ for all $t \geq 0$. Moreover, all of \eqref{eq:T_up}, \eqref{eq:q_up}, and \eqref{eq:s-ineq} hold.  
\end{theorem} 


\section{Non-overshooting Regulation to $s_{\rm r}$
under Higher-Order Actuator Dynamics} \label{sec:high} 

We consider the actuator dynamics whose order is one higher than \eqref{eq:stefan-actuator}, i.e., a double integrator in the actuation path:
\begin{align} \label{eq:stefan-high-act-1}
    \dot q_{\rm c}(t) &= p(t), \\
    \dot p(t) &= U(t) , \label{eq:stefan-high-act-2}
\end{align}
where $p(t)$ is additional state variable defined as the heat flux {\em rate} (velocity), and the input $U(t)$ is {\em heating  acceleration}.  
By the non-overshooting control developed in \eqref{eq:nonover}, we set 
\begin{align}
    p_{\rm nonov}(t) = c_1 c_2 h_1(t) - (c_1 + c_2) h_2 (t), 
\end{align}
where $c_1$ satisfies \eqref{eq:c1_condition}, and $c_2 > 0$. 
Since the relative degree of $h_1$ and $h_2$ is now  3 and 2, respectively, an additional CBF for each must be constructed. Let $h_4$ and $h_5$ be the CBFs defined by 
\begin{align}
    h_4(t) & = p_{\rm nonov}(t) - p(t), \\
   \label{eq:h5-def}     h_5(t) & = p(t) + c_1 h_2(t)= p(t) + c_1 q_{\rm c}(t).
\end{align}
Then, the non-overshooting control is designed so that  
\begin{align}
\label{eq:high-ddth1}    \dot h_1 &= -c_1 h_1 + h_3 , \\
   \label{eq:high-ddth2} \dot h_2 &= - c_1 h_2 + h_5, \\
    \label{eq:high-ddth3} \dot h_3 &= - c_2 h_3 + h_4,  \\
    \dot h_4 &= - c_3 h_4,  \label{eq:high-ddth4}\\
    \dot h_5 & = -c_2 h_5 + c_3 h_4,  \label{eq:high-ddth5-ineq} 
\end{align}
hold. Since $h_3, h_4, h_5$ are defined so \eqref{eq:high-ddth1}--\eqref{eq:high-ddth3} hold, it suffices to design the control input to satisfy \eqref{eq:high-ddth4} and \eqref{eq:high-ddth5-ineq}. To satisfy \eqref{eq:high-ddth4}, the control law is given as 
\begin{align}
    U(t) &= c_1 c_2 c_3 h_1(t) - (c_1 c_2 + c_1 c_3 + c_2 c_3 ) h_2 (t) \notag\\
    & \quad - (c_1 + c_2 + c_3) p (t).  \label{eq:high-nonover}
\end{align}
We require the positivity of all CBFs at $t=0$. The properties $h_1(0) \geq 0$ and $h_2(0) \geq 0$ hold with Assumptions \ref{ass:h2} and \ref{ass:h1}. From the conditions  $h_3(0) \geq 0$, $h_4(0) \geq 0$, and $h_5(0) \geq 0$, the following conditions on the gain parameters arise: 
\begin{align} \label{eq:high-c1-cond}
    c_1 & \geq \max \left\{\fr{q_{\rm c}(0)}{\sigma(0)},  -\fr{p(0)}{q_{\rm c}(0)} \right\} >0, \\
    \label{eq:high-c2-cond} 
    c_2 & \geq \fr{p(0) + c_1 q_{\rm c}(0)}{c_1 \sigma(0) - q_{\rm c}(0)} >0 . 
\end{align}

Let us now examine the chained structure \eqref{eq:high-ddth1}--\eqref{eq:high-ddth5-ineq}. A clearer ordering of these five subsystems is
\begin{equation}
    \dot h_4 = - c_3 h_4
\end{equation}
\vspace*{-.7cm}
\begin{align}
    \dot h_3 &=- c_2 h_3+ h_4, & \dot h_5&= -c_2 h_5+ c_3 h_4\\
    \dot h_1&=- c_1 h_1+ h_3, & \dot h_2&= -c_1 h_2+ h_5.
\end{align}
For even more clarity, we give the following flow diagram among these five positive subsystems:
\begin{equation}
    h_4 
    \begin{array}{lcccl}
   & h_3 & \longrightarrow & h_1\\
     \nearrow \\
     \\
    \searrow 
    \\ & h_5 & \longrightarrow & h_2
    \end{array}
\end{equation}

\begin{theorem}
Let $s(0)$ and $T(x,0)$ satisfy Assumptions \ref{ass:initial}--\ref{ass:h1}. Consider the closed-loop system \eqref{eq:stefanPDE}--\eqref{eq:stefanIC}, \eqref{eq:stefanODE}, and \eqref{eq:stefan-high-act-1} and \eqref{eq:stefan-high-act-2} with the non-overshooting control law \eqref{eq:high-nonover}, where the gain parameters satisfy \eqref{eq:high-c1-cond}, \eqref{eq:high-c2-cond}, and $c_3 > 0$. Then, all CBFs defined as \eqref{eq:h1_def}--\eqref{eq:h_def} satisfy the positivity constraints. This means, in particular, that, for all $q_{\rm c}(0)>0$, and for all real $p(0)$, the interface $s(t)$ does not exceed $s_{\rm r}$, the temperature $T(x,t)$ does not drop below $T_{\rm m}$ at any position $x$ between $0$ and $s(t)$, and the heat flux $q_{\rm c}(t)$ never takes a negative value. 
Furthermore, the closed-loop system is  exponentially stable at the equilibrium $s=s_{\rm r}, T(x,\cdot) \equiv 0, q_{\rm c} =p= 0$,  in the sense of the spatial $L_2$-norm, for all initial conditions in the safe set, i.e., globally. In other words, there exist positive constants $M>0$ and $b>0$ such that the following norm estimate holds: 
\begin{align}
    \Phi(t) \leq M \Phi(0) e^{-b t},  
\end{align}
where 
\begin{align}
\Phi(t):= || T[t] - T_{\rm m} ||^2 +(s(t) - s_{\rm r})^2 + q_{\rm c}(t)^2 + p(t)^2 .
\end{align} 
\end{theorem}

While in this section we pursued just a non-overshooting design for regulation to the barrier, it is straightforward to also design QP safety filters like in Sections \ref{sec:QP} and \ref{sec:above} for the system with an extra integrator, treated in this section. 

Furthermore, there is no obstacle except further algebra for extending the non-overshooting and QP safety filter designs from the single and double integrator models to a full integrator chain model with the heat flux as its first state,
\begin{align}
    \dot q_c(t) & = p_1(t), \\
    \dot p_{i}(t) & = p_{i+1}(t), \quad \forall i = 1, \dots, n-2\\
    \dot p_{n-1}(t) & = U(t). 
\end{align}
The most interesting new element in the general case are the extra conditions on the gains $c_1, \ldots, c_n$. There would be $n$ conditions for the gain parameter $c_1$, $n-1$ conditions for $c_2$, and, continuing in the same pattern, finally, one condition for $c_{n-1}$, while $c_n$ would be an arbitrary positive constant.  


\section{Safety for the Two-phase Stefan system under disturbance} \label{sec:twophase} 

\begin{figure}[t]
\centering
\includegraphics[width=\linewidth]{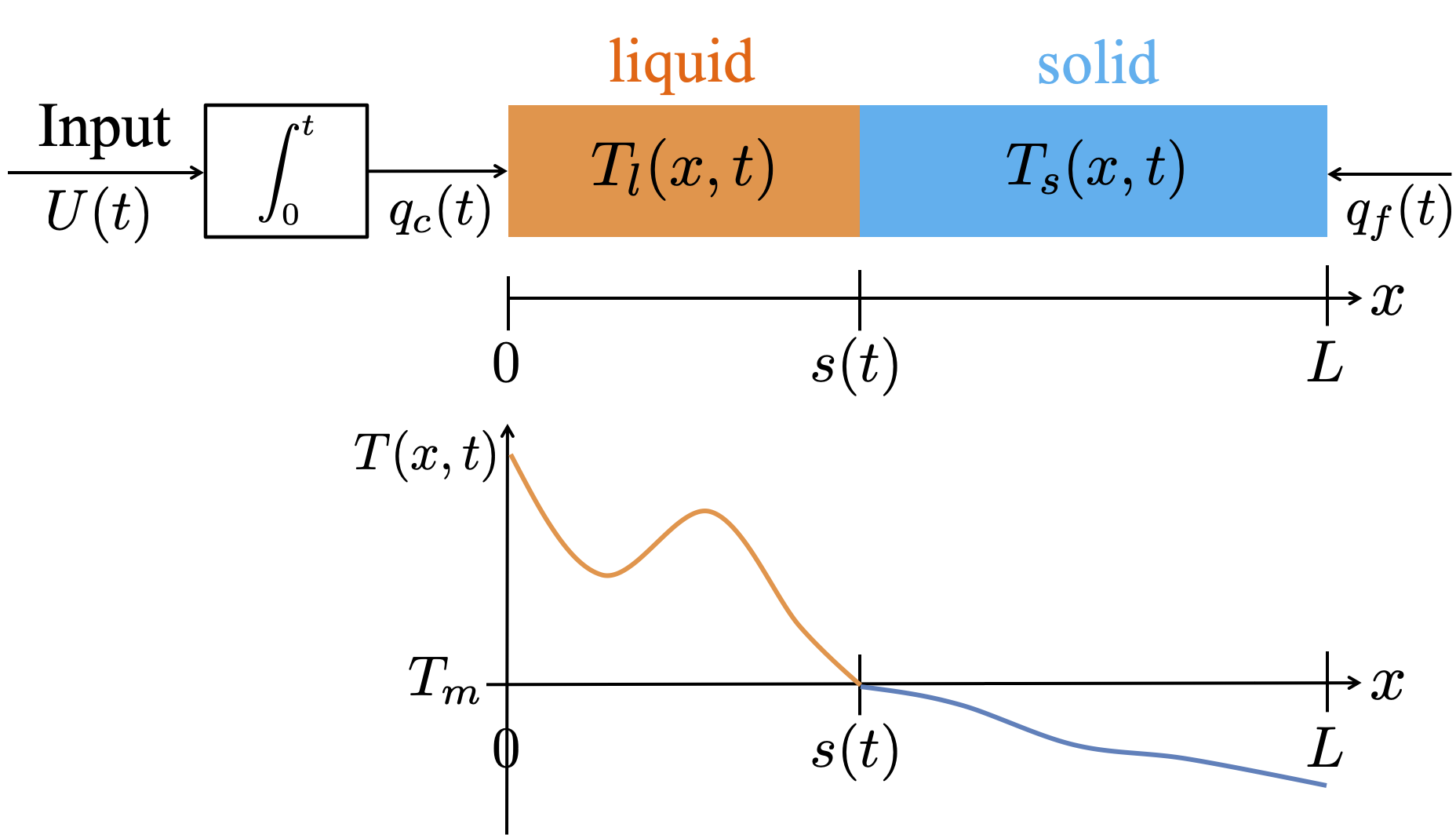}\\
\caption{Schematic of the two-phase Stefan problem with disturbance at $x = L$.}
\label{fig:2phstefan}
\end{figure}

In this section, we extend the safety design to the "two-phase" Stefan problem, where the interface dynamics is affected by a heat loss modeled by the temperature dynamics in the solid phase, following the work in \cite{koga2019twophase}. Additionally, the solid phase temperature is affected by a heat loss caused at the end boundary of the solid phase, which serves as a disturbance in the system. The safety constraint and closed-loop analysis for such two-phase Stefan system with disturbance have been studied in \cite{koga2021towards}, by means of Input-to-State Stability (ISS). In this paper, we further incorporate the actuator dynamics, and tackle the safety verification by utilizing CBFs and nonovershooting regulation. The configuration of the two-phase Stefan problem is depicted in Fig. \ref{fig:2phstefan}. 

\subsection{Model and Constraint}

The governing equations are described by the following coupled PDE-ODE-PDE system: 
\begin{align}\label{2ph-sys1}
 \fr{\pa T_{{\rm l}}}{\pa t}(x,t) &=\alpha_{{\rm l}}  \fr{\pa^2 T_{{\rm l}}}{\pa x^2}(x,t),  \hspace{1mm} \textrm{for} \hspace{1mm} t > 0,  \hspace{1mm} 0< x< s(t),\\
\label{2ph-sys2} \fr{\pa T_{{\rm l}}}{\pa x}(0,t) &= -q_{{\rm c}}(t)/k_{{\rm l}}, \hspace{1mm} T_{{\rm l}}(s(t),t) =T_{{\rm m}}, \textrm{ for } t >0, \\
\label{2ph-sys3} \fr{\pa T_{{\rm s}}}{\pa t}(x,t) &=\alpha_{{\rm s}}  \fr{\pa^2 T_{{\rm s}}}{\pa x^2}(x,t), \hspace{1mm} \textrm{for} \hspace{1mm} t > 0,  \hspace{1mm} s(t)< x< L, \\
\label{2ph-sys4} \fr{\pa T_{{\rm s}}}{\pa x}(L,t) &= - q_{{\rm f}}(t)/k_{{\rm s}},\hspace{1mm} T_{{\rm s}}(s(t),t) =T_{{\rm m}},  \textrm{ for } t >0, \\
\label{2ph-actuator} \dot q_{\rm c}(t) & = U(t), \\
\label{2ph-sys5} \gam \dot{s}(t) &= - k_{{\rm l}} \fr{\pa T_{{\rm l}}}{\pa x}(s(t),t)+k_{{\rm s}} \fr{\pa T_{{\rm s}}}{\pa x}(s(t),t),
\end{align}
where $\gamma = \rho_{{\rm l}} \Delta H^*$, and all the variables denote the same physical value with the subscript "l" for the liquid phase and "s" for the solid phase, respectively. The boundary condition of the solid phase temperature given in \eqref{2ph-sys4} is affected by an unknown heat loss, where $q_{\rm f}(t) \geq 0$ denotes the magnitude of the cooling heat flux at the side of the solid phase, which serves as a disturbance of the system. The solid phase temperature must be lower than the melting temperature, which serves as one of the conditions for the model validity. Namely, we require
\begin{align}\label{valid1-2ph}
T_{{\rm l}}(x,t) \geq& T_{{\rm m}}, \quad \forall x\in(0,s(t)), \quad \forall t>0, \\
\label{valid2-2ph}T_{{\rm s}}(x,t) \leq& T_{{\rm m}}, \quad \forall x\in(s(t),L), \quad \forall t>0, \\
\label{valid3-2ph} 0< &s(t)<L, \quad \forall t>0. 
\end{align}

The following assumption on the initial data \\$(T_{{\rm l},0}(x), T_{{\rm s},0}(x), s_0) := (T_{{\rm l}}(x,0), T_{{\rm s}}(x,0), s(0))$ is imposed.  
\begin{assum}\label{initial-2ph} 
$0<s_0<L$,  $T_{{\rm l},0}(x) \in C^0([0, s_0];[T_{\rm m}, +\infty))$, $T_{{\rm s},0}(x) \in C^0([s_0, L];( - \infty, T_{\rm m}])$, and $T_{{\rm l},0}(s_0) = T_{{\rm s},0}(s_0) = T_{\rm m}$. Also, there exists constants $\bar T_{\rm l}, \bar T_{\rm s}, \eta_{\rm l}, \eta_{\rm s} >0$ such that 
\begin{align}
0 \leq T_{{\rm l},0}(x) - T_{\rm m} \leq \bar T_{\rm l} \left( 1 - \exp (L \eta_{\rm l} \alpha_{\rm l}^{-1} (x - s_0)\right) 
\end{align} for $x \in [0, s_0]$ and 
\begin{align}
- \bar T_{\rm s} \left( 1 - \exp (L \eta_{\rm s} \alpha_{\rm s}^{-1} (x - s_0)\right)  \leq T_{{\rm s},0}(x) - T_{\rm m} \leq 0 
\end{align}
for $x \in [s_0, L]$. 
 \end{assum}
 The following lemma is provided to ensure the conditions of the model validity. 

\begin{lemma}\label{solidvalid}
Let Assumption \ref{initial-2ph} hold, $q_{{\rm c}}(t)$ and $q_{\rm f}(t)$ be bounded nonnegative continuous functions $q_{\rm c} \in C^0(\R_+; [0,\bar q_{\rm c}))$, $q_{\rm f} \in C^0(\R_+; [0,\bar q_{\rm f}))$ for some $\bar q_{\rm c}, \bar q_{\rm f} >0$, and
\begin{align} \label{cond:2phposed}
\max \left\{ \frac{k_{\rm l} \ep_{\rm l}}{\alpha_{\rm l}} \left( 1 + \frac{\alpha_{\rm l}}{L^2 \eta_{\rm l}} \right) ,  \frac{k_{\rm s} \ep_{\rm s}}{\alpha_{\rm s}} \left( 1 + \frac{\alpha_{\rm s}}{L^2 \eta_{\rm s}} \right) \right\} < \frac{\gamma}{4} , 
\end{align} 
hold, where 
\begin{align}
\ep_{\rm l} := \max\left\{\bar T_{\rm l}, \bar q_{\rm c} L k_{\rm l}^{-1} \right\}, \hspace{2mm} \ep_{\rm s} := \max\left\{\bar T_{\rm s}, \bar q_{\rm f} L k_{\rm s}^{-1} \right\}.
\end{align} 
Furthermore, suppose it holds 
\begin{align} \label{2ph-validcondition} 
0< \gam s_{\infty} + \int_0^{t} (q_{{\rm c}}(\tau) - q_{\rm f}(\tau))  {\rm d}\tau < \gam L ,
\end{align}
for all $t \geq 0$, where 
\begin{align}
s_{\infty}:= & s_0 + \frac{k_{{\rm l}}}{\alpha_{{\rm l}} \gam} \int_{0}^{s_0} (T_{{\rm l},0}(x) - T_{{\rm m}}) {\rm d}x  \notag\\
& + \frac{k_{{\rm s}}}{\alpha_{{\rm s}}\gam} \int_{s_0}^{L} (T_{{\rm s},0}(x) - T_{{\rm m}}) {\rm d}x. \label{eq:2ph-sinf} 
\end{align}
Then there exists a unique solution to \eqref{2ph-sys1}--\eqref{2ph-sys5} which satisfies \eqref{valid1-2ph}--\eqref{valid3-2ph}. 
\end{lemma}

Lemma \ref{solidvalid} is proven in \cite{Cannon71flux} (Theorem 1 in p.4 and Theorem 4 in p.8) by employing the maximum principle. The variable $s_{\infty}$ is the final interface position $s_{\infty} = \lim_{t \to \infty} s(t)$ under the zero heat input $q_{{\rm c}}(t) \equiv q_{\rm f}(t) \equiv 0$ for all $t \geq 0$.  For \eqref{2ph-validcondition} to hold for all $t \geq 0$, we at least require it to hold at $t = 0$, which leads to the following assumption. 
\begin{assum}\label{ass:E0}
$s_{\infty}$ given by \eqref{eq:2ph-sinf} satisfies $0< s_{\infty} < L $. 
\end{assum}

\begin{figure}[t]
\centering
\includegraphics[width=\linewidth]{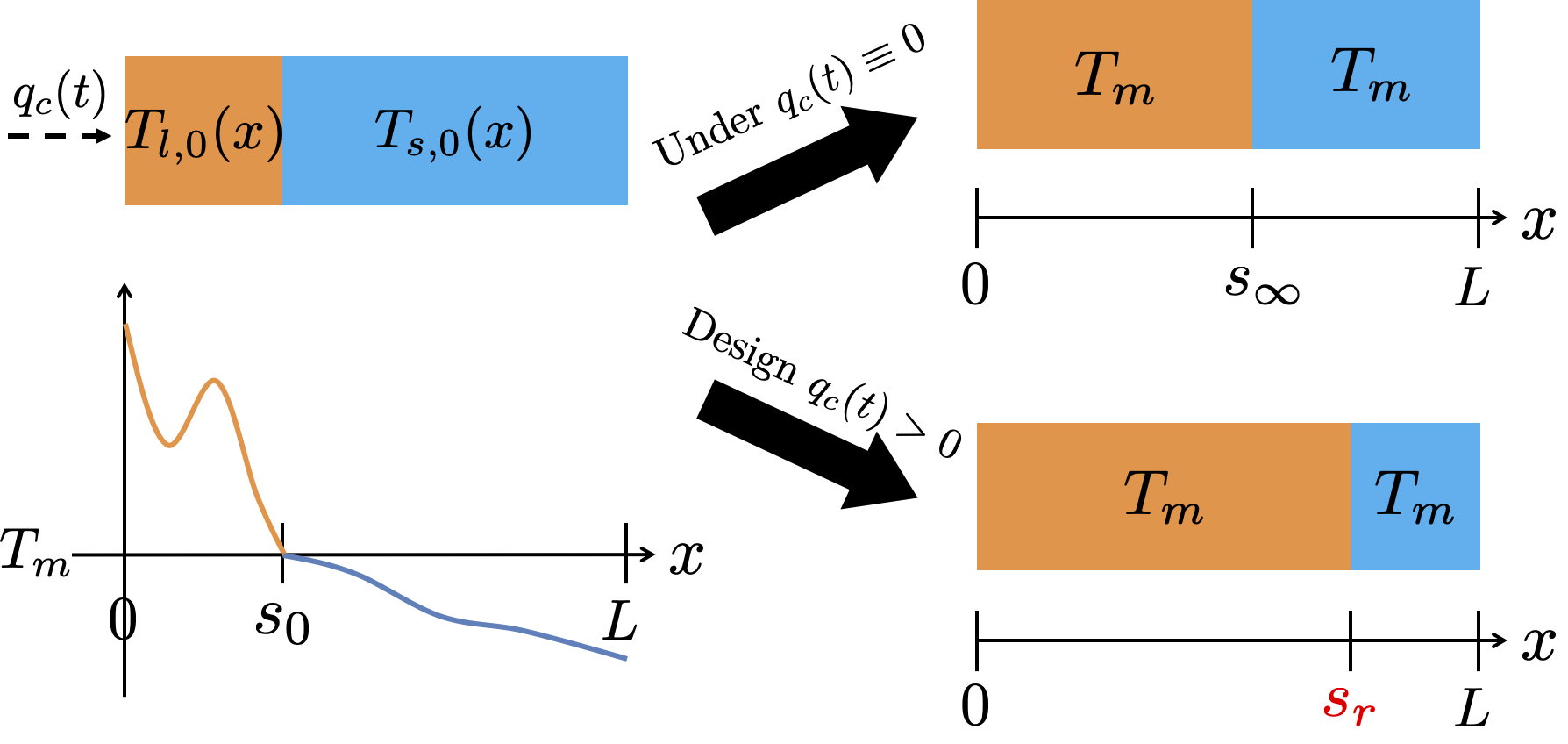}\\
\caption{A graphic interpretation of $s_{\infty}$ and $s_{\rm r}$ under $q_{\rm f} \equiv 0$.}
\label{fig:2phstefan-obj}
\end{figure}

Furthermore, we impose the restriction for the setpoint is given as follows. 
\begin{assum}\label{assum2-2ph}
The setpoint $s_{\rm r}$ satisfies $s_{\infty} < s_{{\rm r}} < L$. 
\end{assum}

Physically, Assumption \ref{ass:E0} states that neither phase disappears under $q_{\rm c}(t) \equiv q_{\rm f}(t) \equiv 0$, and Assumption \ref{assum2-2ph} states that the choice of the setpoint for the melting is far beyond $s_{\infty}$ from the heat input. A graphical illustration of the assumptions can be seen in Fig. \ref{fig:2phstefan-obj}. 

\subsection{Nonovershooting Regulation and Guaranteed Safety} 

Due to the addition of the solid phase temperature dynamics, the energy deficit $\sigma$ in \eqref{eq:h1_def} is reformulated. We consider the following CBFs for the two-phase problem: 
\begin{align}
     h_1(t) :&= \sigma(t)  \nonumber\\
    & =  -\left(\frac{k_{{\rm l}}}{\alpha_{{\rm l}}} \int_{0}^{s(t)} (T_{{\rm l}}(x,t) - T_{{\rm m}}) {\rm d}x \right. \notag\\
&\left. + \frac{k_{{\rm s}}}{\alpha_{{\rm s}}} \int_{s(t)}^{L} (T_{{\rm s}}(x,t) - T_{{\rm m}}) {\rm d}x + \gamma (s(t) - s_{{\rm r}})\right), \label{eq:2ph-h1_def} \\
h_2(t) &= q_{\rm c}(t), \\
h_3(t) &= - q_{\rm c}(t)+c_1 \sigma(t) , \\
h_{\rm l}(x,t) &= T_{\rm l}(x,t) - T_{\rm m}, \\
h_{\rm s}(x,t) &= T_{\rm m} - T_{\rm s}(x,t), \label{eq:2ph-hs_def} 
\end{align}

Then, it holds $\dot h_1 = - h_2 + q_{\rm f}(t)$, thereby the non-overshooting regulation is designed in a same manner as in Section \ref{sec:nonov}, which makes the following relationship to hold: 
\begin{align} 
\label{eq:2ph-ddth1}    \dot h_1 &= -c_1 h_1 + h_3 + q_{\rm f}(t), \\
   \label{eq:2ph-ddth2} \dot h_2 &= - c_1 h_2 + c_2 h_3, \\
    \label{eq:2ph-ddth3} \dot h_3 &= - c_2 h_3 + c_1 q_{\rm f}(t). 
\end{align}
Namely, the nonovershooting regulation is designed as 
\begin{align} \label{eq:2ph-nonover} 
    U^*(\sigma,q_{\rm c}) =  - (c_1 + c_2) q_{\rm c} + c_1 c_2 \sigma  . 
\end{align}
To satisfy the inequality \eqref{cond:2phposed}, we impose the following assumption. 
\begin{assum} \label{2ph:ass:add} 
$ q_{\rm f}(t)\in C^0(\R_+; [0,\bar q_{\rm f}))$ for some $\bar q_{\rm f}>0$ satisfying 
\begin{align} \label{eq:2ph-qbar-cond} 
\bar q_{\rm f} < \min \left\{q_{\rm c}(0) + c_1 \gam s_{\infty}, \fr{c_1 c_2 }{c_1 + c_2} \gam s_{\rm r} \right\},     
\end{align}
and the inequality \eqref{cond:2phposed} holds with 
\begin{align}
\bar q_{\rm c} = \max \left\{q_{\rm c}(0), \fr{c_2}{c_1} (c_1 \sigma(0) - q_{\rm c}(0)), \bar q_{\rm f} \right\}.
\end{align}
\end{assum}
With $q_{\rm f} \geq 0$ in Assumption \ref{2ph:ass:add}, and with the gain condition 
\begin{align} \label{eq:2ph-gain} 
    c_1 \geq \fr{q_{\rm c}(0)}{\sigma(0)}, 
\end{align}
the relationship \eqref{eq:2ph-ddth1}--\eqref{eq:2ph-ddth3} leads to  
\begin{align}
    h_1 \geq 0, \quad h_2 \geq 0, \quad h_3 \geq 0.
\end{align}
Moreover, in the same manner as in Section \ref{sec:above}, an upper bound of the solution to \eqref{eq:2ph-ddth1}--\eqref{eq:2ph-ddth3} is also obtained, which ensures that $q_{\rm c}(t) \leq \bar q_{\rm c}$. Therefore,  with Assumption \ref{2ph:ass:add}, the condition \eqref{cond:2phposed} is satisfied. 

It remains to ensure the condition \eqref{2ph-validcondition}. Taking the time integration to the relation $\dot h_1 = - q_{\rm c}(t) + q_{\rm f}(t)$, and by $h_1(0) = \gam s_{\rm r} - \gam s_{\infty}$, one can obtain 
\begin{align}
    \gam s_{\infty} + \int_0^t (q_{\rm c}(\tau) - q_{\rm f}(\tau)) d\tau = \gam s_{\rm r}- h_1(t). 
\end{align}
Since $h_1 \geq 0$ is already ensured, with Assumption \ref{assum2-2ph}, the right inequality in \eqref{2ph-validcondition} is satisfied. Moreover, by the relation \eqref{eq:2ph-ddth1}--\eqref{eq:2ph-ddth3}, the upper bound of $h_1$ is shown as 
\begin{align}
    h_1(t) \leq \max \left\{h_1(0), \fr{1}{c_1} \left(\max \left\{h_3(0), \fr{c_1}{c_2} \bar q_{\rm f} \right\} + \bar q_{\rm f} \right) \right\} .  
\end{align}
Hence, with \eqref{eq:2ph-qbar-cond} in Assumption \ref{2ph:ass:add}, one can deduce that the left inequality in \eqref{2ph-validcondition} is also satisfied. Thus, applying Lemma \ref{solidvalid}, the safety constraint \eqref{valid1-2ph}--\eqref{valid1-2ph} is satisfied under the nonovershooting design in \eqref{eq:2ph-nonover}. We state the following theorem. 

\begin{theorem} \label{thm:2ph-nonover} 
Let $s(0)$ and $T_{\rm l}(x,0)$, and $T_{\rm s}(x,0)$ satisfy Assumption \ref{initial-2ph}-\ref{assum2-2ph}, and $q_{\rm f}(t)$ satisfy Assumption \ref{2ph:ass:add}. Consider the closed-loop system \eqref{2ph-sys1}--\eqref{2ph-sys5} with the non-overshooting control law \eqref{eq:2ph-nonover}, where the gain parameters satisfy \eqref{eq:2ph-gain} and $c_2 \geq 0$. Then, all CBFs defined as \eqref{eq:2ph-h1_def}--\eqref{eq:2ph-hs_def} satisfy the positivity constraints. This means, in particular, that, for all $q_{\rm c}(0)>0$, the interface $s(t)$ remains inside $(0,L)$, the liquid temperature $T_{\rm l}(x,t)$ does not drop below $T_{\rm m}$ at any position $x$ between $0$ and $s(t)$, the solid temperature $T_{\rm s}(x,t)$ does not go above $T_{\rm m}$ at any position $x$ between $s(t)$ and $L$, and the heat flux $q_{\rm c}(t)$ never takes a negative value. 
Furthermore, the closed-loop system is  exponentially Input-to-State Stability (ISS) at the equilibrium $s=s_{\rm r}, T_{\rm l}(x,\cdot) \equiv T_{\rm m}, T_{\rm s}(x,\cdot) \equiv T_{\rm m}, q_{\rm c} = 0$,  in the sense of the spatial $L_2$-norm, for all initial conditions in the safe set, i.e., globally. In other words, there exist positive constants $M_1>0$, $M_2 >0$, and $b>0$ such that the following norm estimate holds: 
\begin{align} \label{eq:2ph-norm-estimate} 
    \Phi(t) \leq M_1 \Phi(0) e^{-b t} + M_2 \sup_{\tau \in [0, t]} q_{\rm f}(\tau),  
\end{align}
where 
\begin{align} \label{eq:2ph-Phi-def} 
\Phi(t) &:= || T_{\rm l}[t] - T_{\rm m} ||^2 + || T_{\rm s}[t] - T_{\rm m} ||^2 \notag\\
& \quad +(s(t) - s_{\rm r})^2 + q_{\rm c}(t)^2. 
\end{align}

\end{theorem}

The safety is already proven in this section. The ISS proof is identical to the steps performed in \cite{koga2021towards} (see Section 5-2), which is omitted in this paper.





\section{Application to Additive Manufacturing} \label{sec:AM} 

\begin{figure}[t]
\centering
\includegraphics[width=0.99\linewidth]{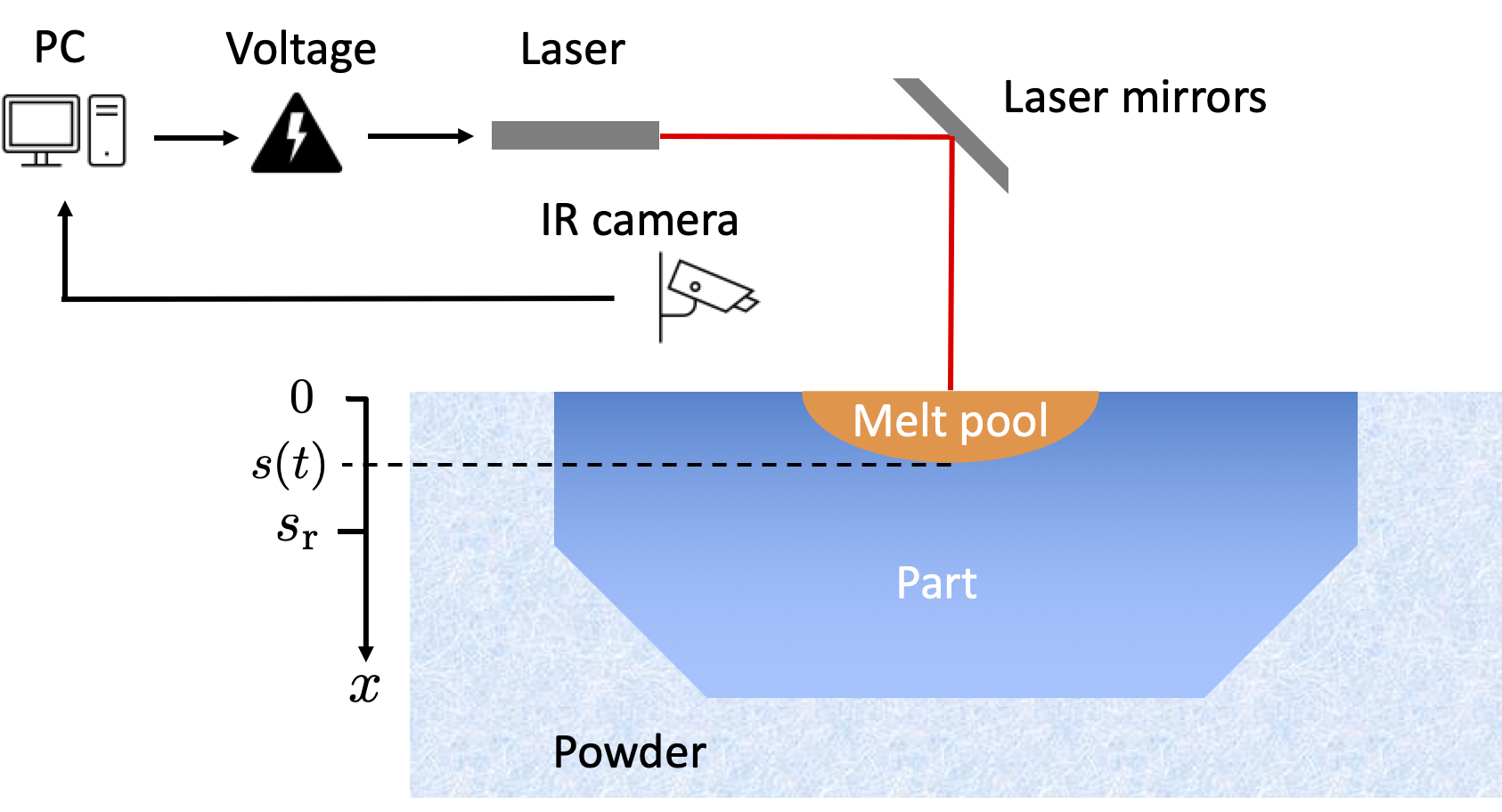}\\
\caption{Schematic of powder bed metal additive manufacturing by selective laser sintering. The safe control should be designed to keep generating positive laser power and to avoid overshoot of the depth of melt pool beyond the desired layer thickness. }
\label{fig:AM}
\end{figure}

\begin{figure*}[t]
\begin{center} 
\subfloat[Interface position converges to  setpoint without overshooting. ]
{\includegraphics[width=2.1in]{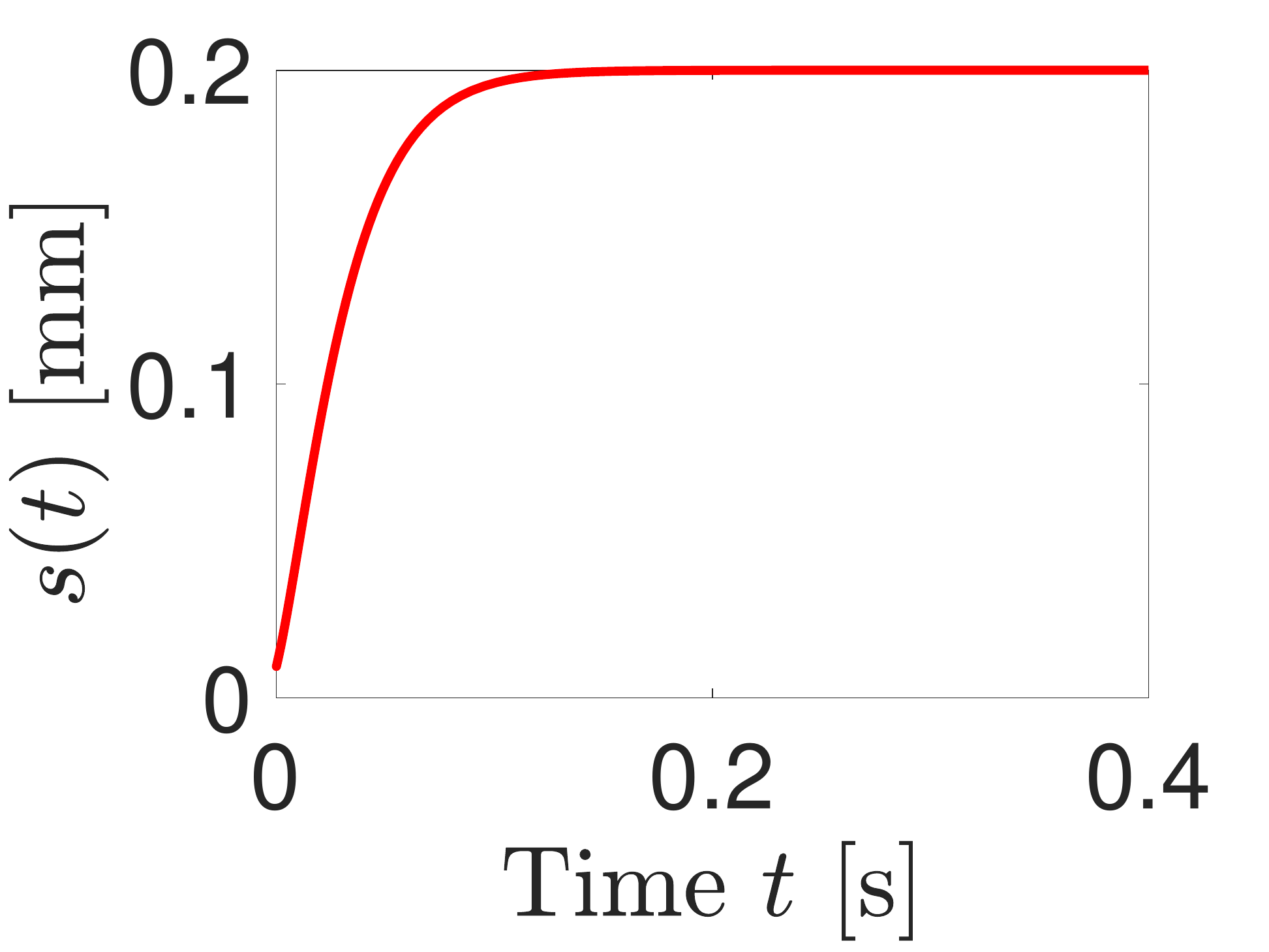}\label{fig:interface}}\hspace{1mm}
\subfloat[Liquid at the location of the laser power $q_{\rm c}$ first warms and then cools, remaining above the melting temperature 1650 Celsius degree. ]
{\includegraphics[width=2.1in]{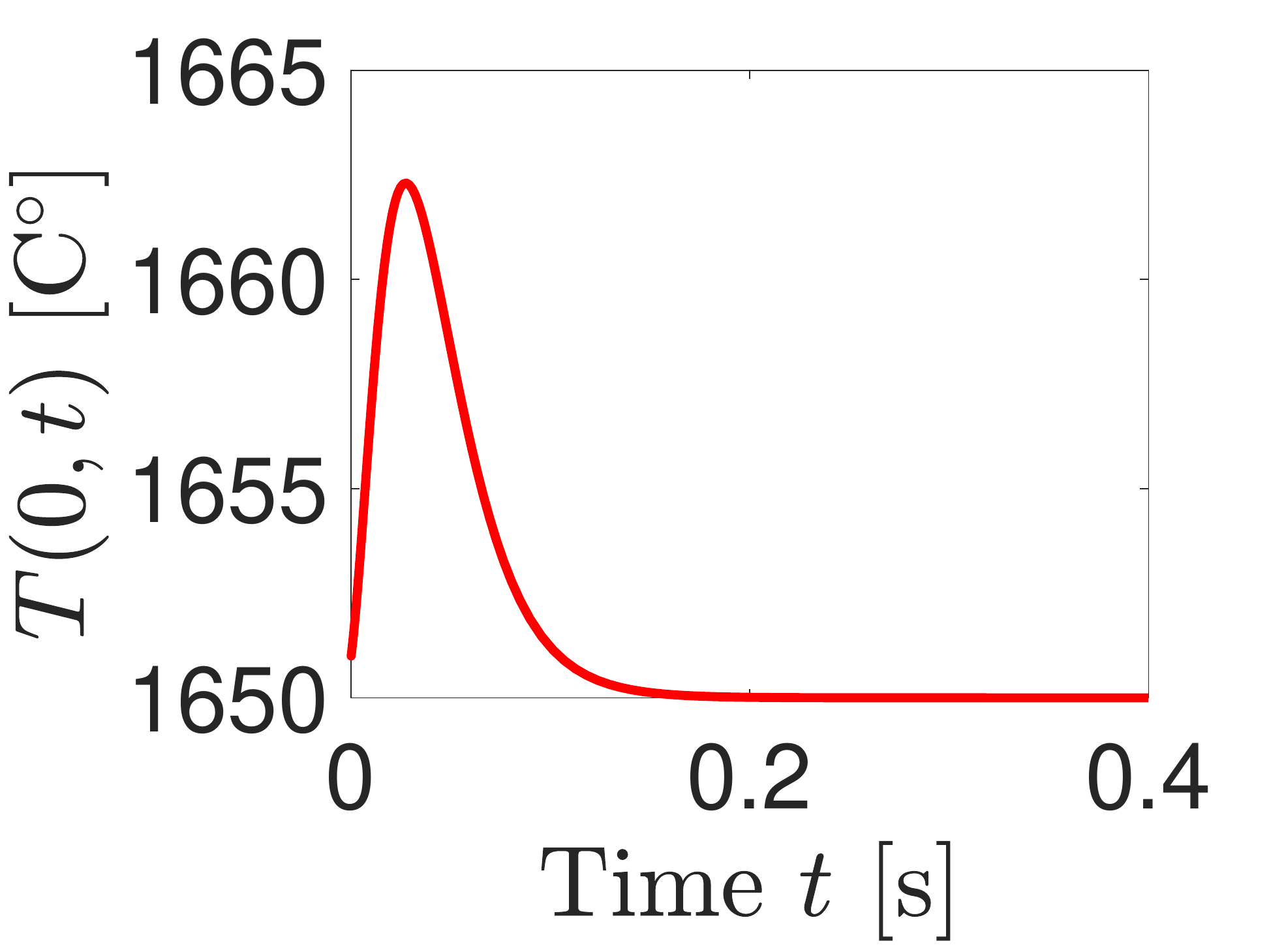}\label{fig:temp}}\hspace{1mm}
\subfloat[Laser power remains positive. ]
{\includegraphics[width=2.1in]{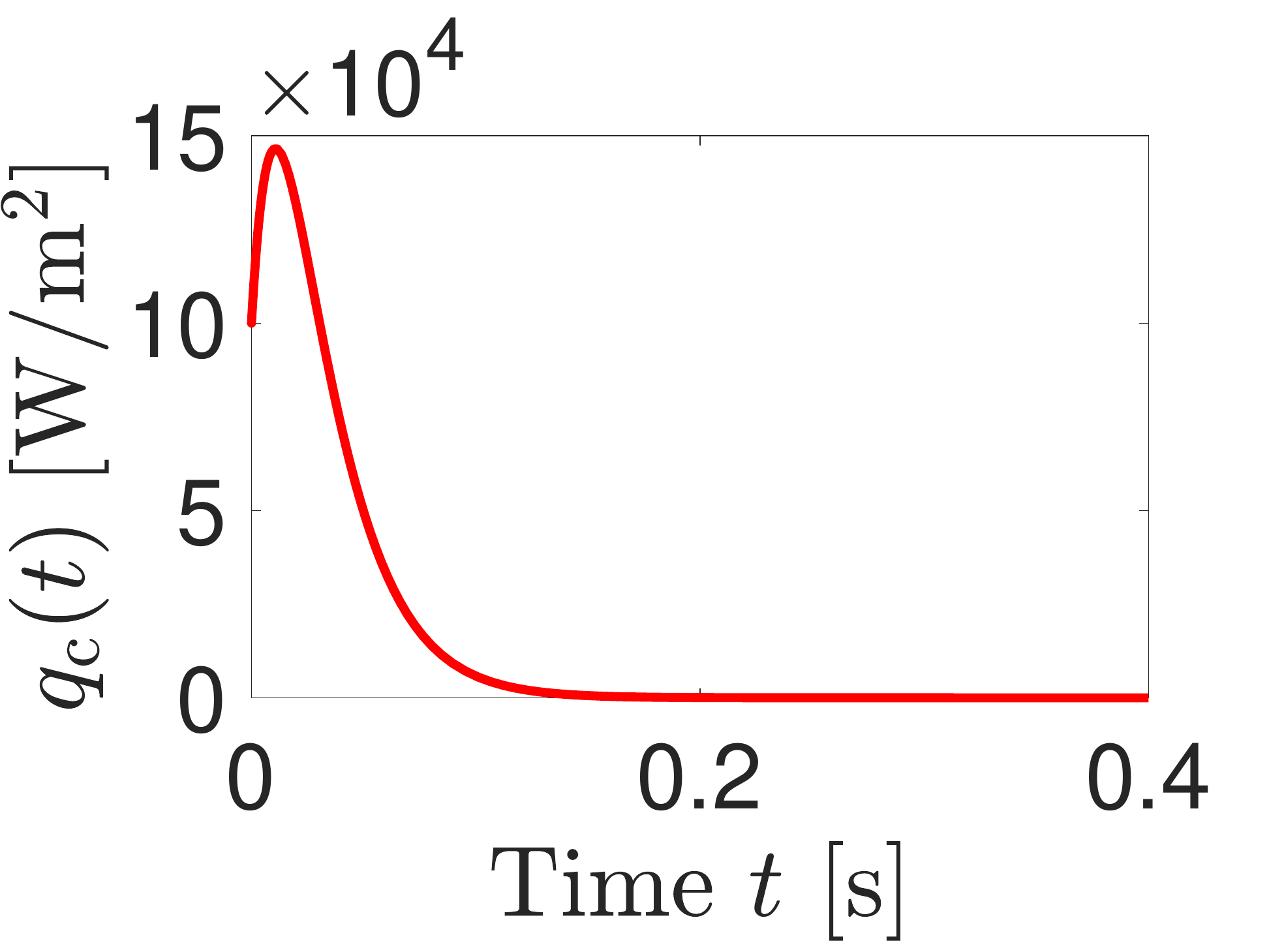}\label{fig:heat}}\\
\subfloat[Energy CBF remains positive. ]
{\includegraphics[width=2.0in]{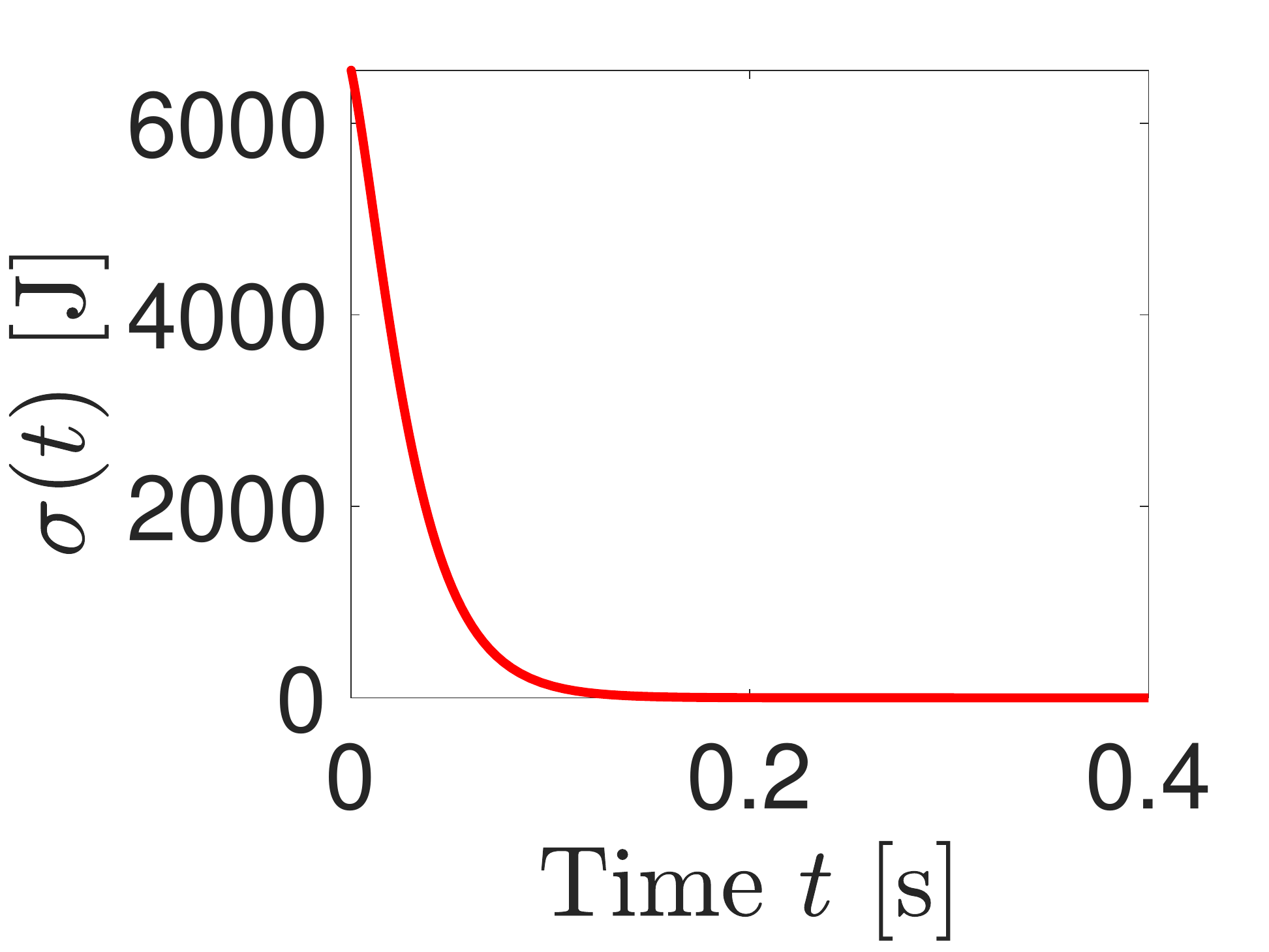}\label{fig:energy}}\hspace{1mm}
\subfloat[Voltage input $U$ achieves non-overshooting regulation of $s$. ]
{\includegraphics[width=2.0in]{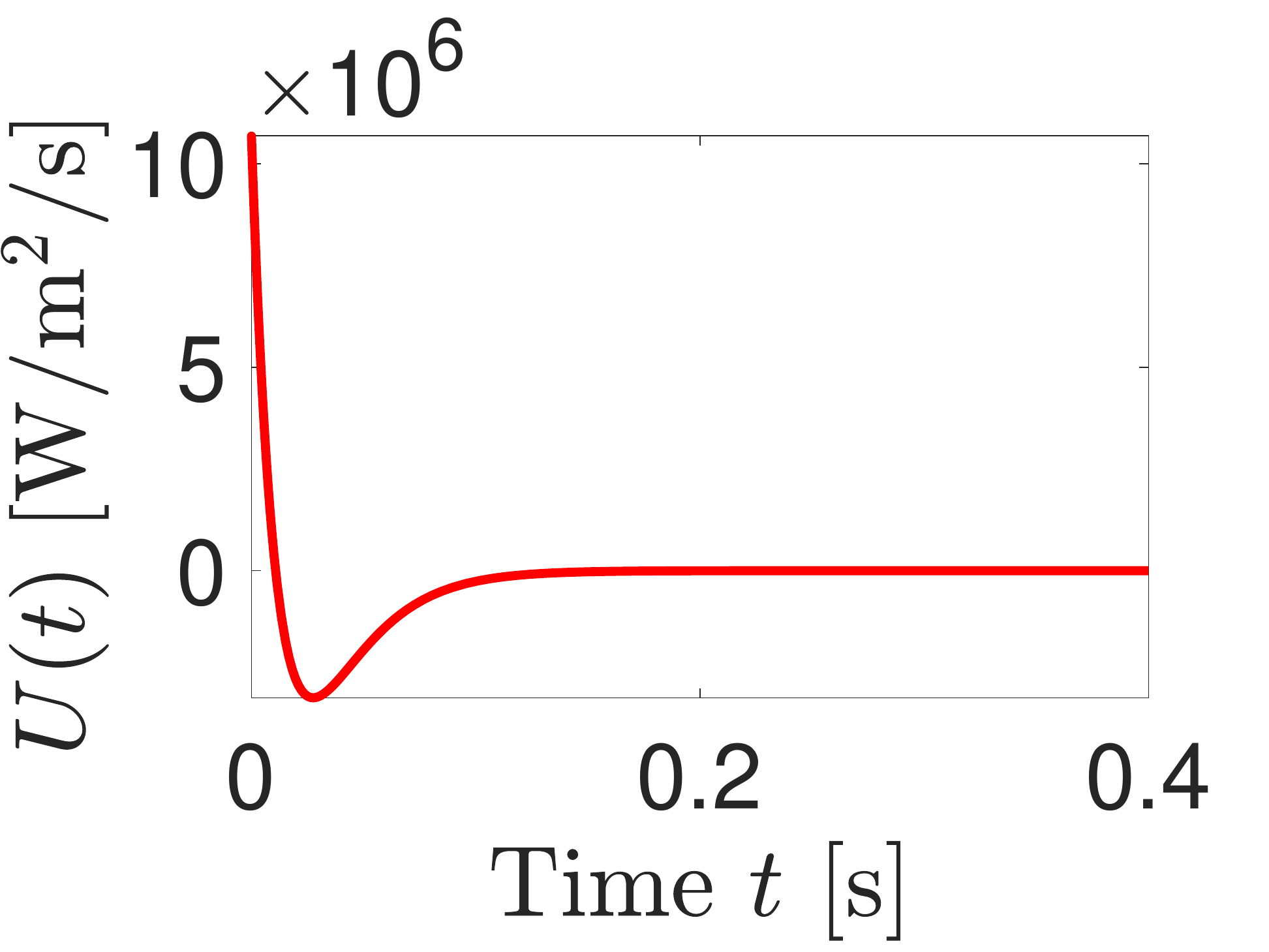}\label{fig:input}}
\subfloat[ Liquid temperature keeps above the melting temperature in all liquid domain. ]
{\includegraphics[width=0.42\linewidth]{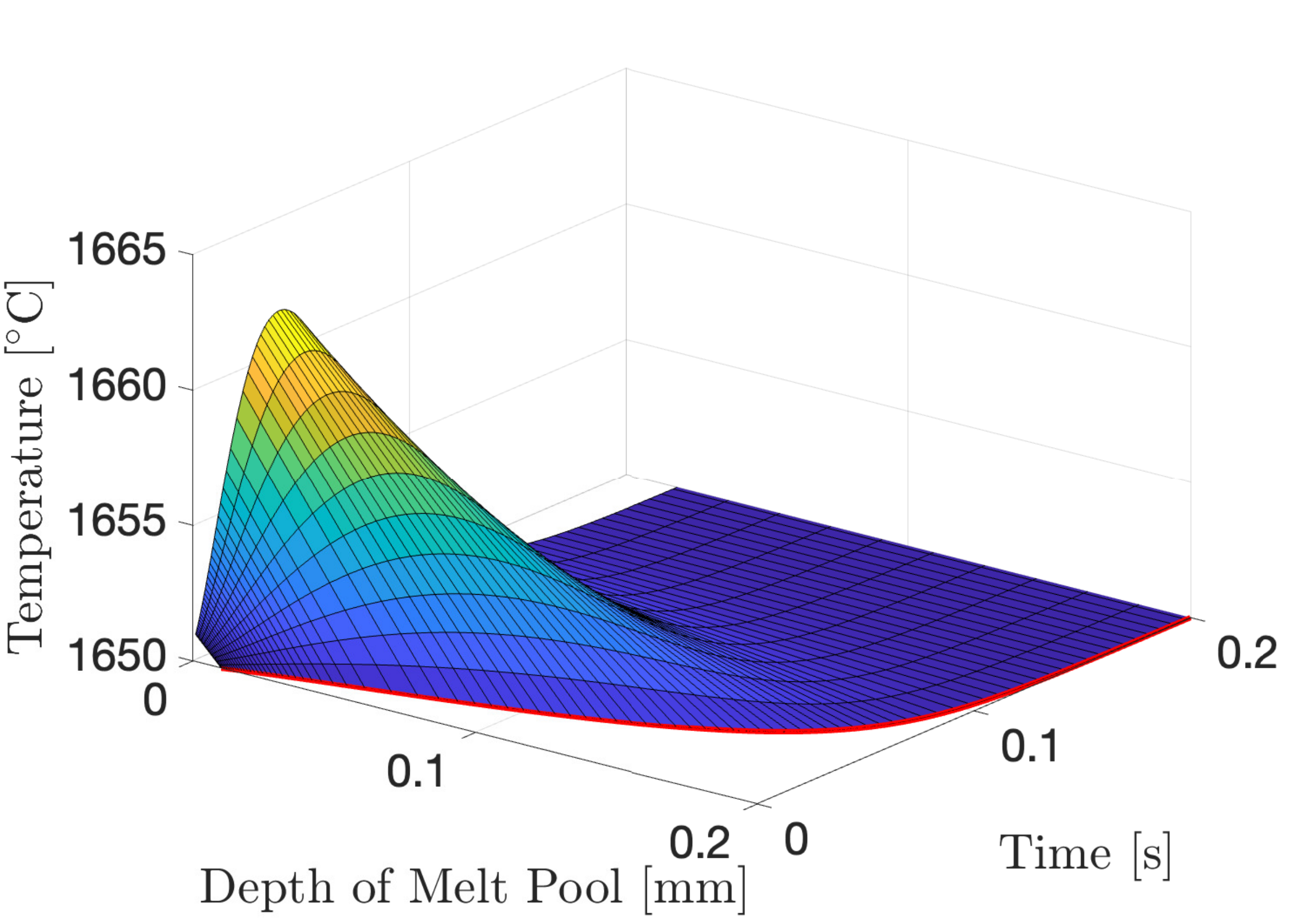}\label{fig:surface}}
\caption{Backstepping control ensures that $s(t)\rightarrow s_{\rm r}$ but prevents $s(t)$ from overshooting $s_{\rm r}$, $T(x,t)$ from undershooting $T_{\rm m}$, and $q_c(t)$ from undershooting zero. This is achieved by non-monotonic $q_{\rm c}(t)$, which overshoots in the positive direction.}
\label{fig:result2}
\end{center} 
\end{figure*}

We apply the safe control methods proposed in Section \ref{sec:nonov} and \ref{sec:QP} to metal additive manufacturing (AM) with selective laser sintering, which has been intensively advanced in the recent decade as observed from the growth in global market \cite{deloitte}. The Stefan model describes the expansion of melt pool inside the powder bed, which is generated by laser input \cite{Zeng12}. We consider controlling the voltage input which manipulates the laser power, for the sake of obtaining the desired depth of the melt pool, which is a thickness of layer. The safe set is the positivity of the laser power and the energy deficit for avoiding the overshooting of the melt pool deeper than the desired thickness. The schematic of the metal AM is depicted in Fig. \ref{fig:AM}. 

We perform the numerical simulation considering a titanium alloy (Ti6Al4V), the physical parameters of which are given in Table 1. The one-dimensional Stefan model is numerically computed by the well-known boundary immobilization method combined with finite difference semi-discretization \cite{kutluay97}.

First, we investigate the performance of the non-overshooting control law \eqref{eq:nonover-fix}. The initial depth  of the melt pool is set to $s_0 =  0.01$ [mm], and the initial temperature profile is set to a linear profile $T_0(x) = \bar T (1 - x / s_0) + T_{\rm m}$ with $\bar T =  1$ [$^\circ$C]. The setpoint position is set as $s_{\rm r} = 0.2$ [mm], which is a reasonable value for layer thickness in metal AM. The control gains are set as $k_1 = 8 q_{\rm c}(0) / \sigma(0) = 64.4$ [/s], $k_2 = 973$ [/s$^2$]. Fig. \ref{fig:result2} depicts the result of the closed-loop response. Fig. \ref{fig:result2} (a) illustrates that the melt pool depth successfully converges to the setpoint monotonically without overshooting, namely, the regulation is achieved. Fig. \ref{fig:result2} (b) shows that the surface temperature of the melt pool is warmed up first and cooled down once the heat is sufficiently added, with maintaining above the melting temperature. Fig. \ref{fig:result2} (c) and (d) depict the two CBFs imposed in the problem, one for an energy deficit and the other for a laser power, both of which satisfy the positivity condition. Fig. \ref{fig:result2} (e) shows the voltage input which does not have any constraint under the non-overshooting control. Fig. \ref{fig:result2} (f) shows 3-D surface plot of the temperature profile in the liquid melt pool, which remains above the melting temperature in all the liquid domain. Hence, the numerical result is consistent with Theorem \ref{thm:nonover} on non-overshooting control in terms of achieving stabilization and safety simultaneously.

	\begin{table}[t]
\caption{Physical properties of Ti6Al4V alloy \cite{mills02}}
\begin{center}
    \begin{tabular}{| l | l | l | }
    \hline
    $\textbf{Description}$ & $\textbf{Symbol}$ & $\textbf{Value}$ \\ \hline
    Density & $\rho$ & 3920 ${\rm kg}\cdot {\rm m}^{-3}$\\ 
    Latent heat of fusion & $\Delta H^*$ & 2.86 $\times $ 10$^5$ ${\rm J}\cdot {\rm kg}^{-1}$ \\ 
    Heat Capacity & $C_{{\rm p}}$ & 830 ${\rm J} \cdot {\rm kg}^{-1}\cdot$ $^\circ {\rm C}^{-1}$  \\  
    Thermal conductivity & $k$ & 32.5 ${\rm W}\cdot {\rm m}^{-1}$  \\
    Melting temperature & $T_{\rm m}$ & 1650 $^\circ$C  \\\hline
    \end{tabular}
\end{center}
\end{table}

\begin{figure*}[t]
\begin{center} 
\subfloat[As a result of commanded heating, interface $s$ settles to constraint $s_{\rm r}$.]
{\includegraphics[width=2.1in]{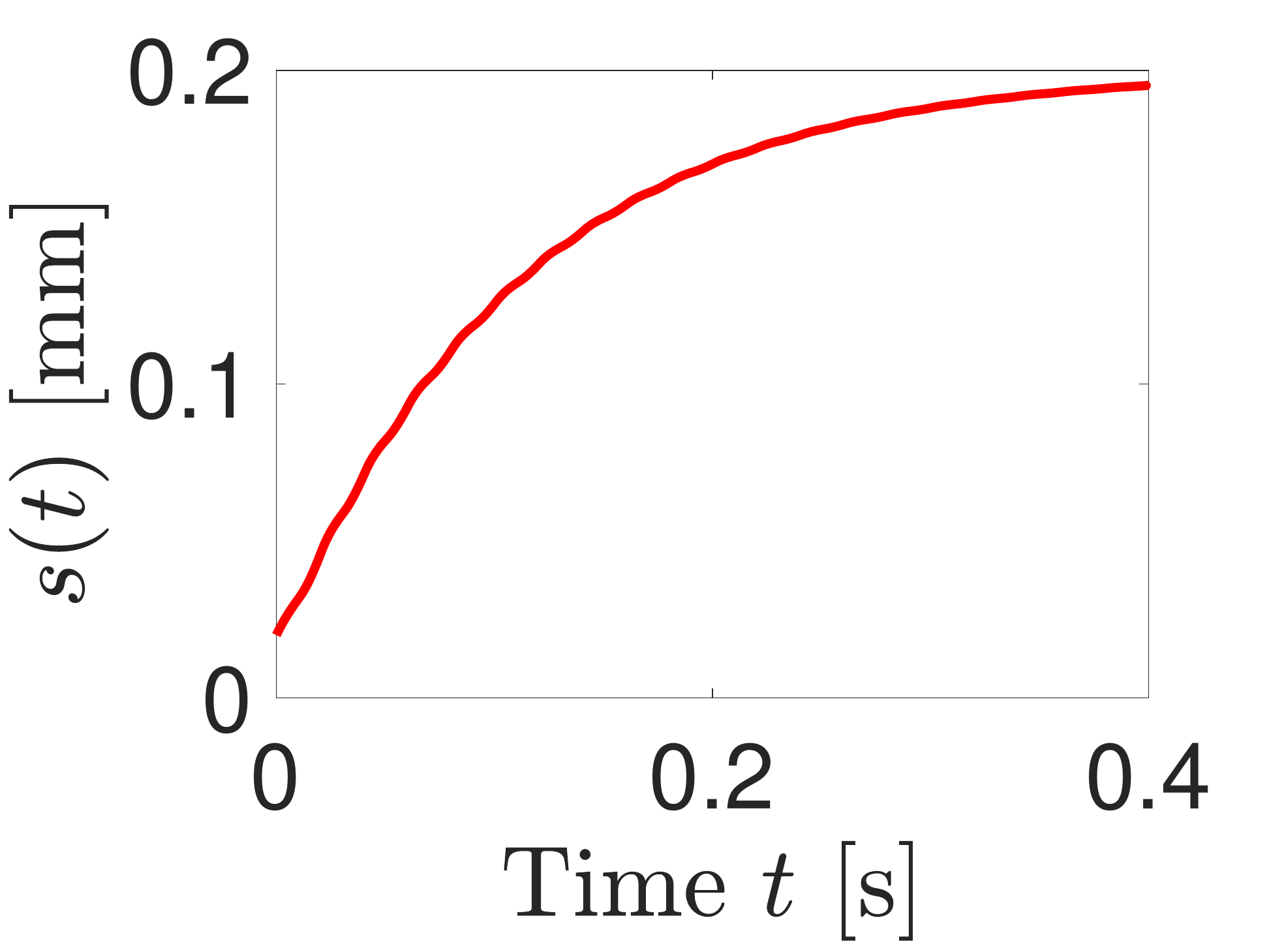}\label{fig:qp_interface}}\hspace{1mm}
\subfloat[Temperature adjacent to heat flux  $q_{\rm c}$ fluctuates but remains above melting. ]
{\includegraphics[width=2.1in]{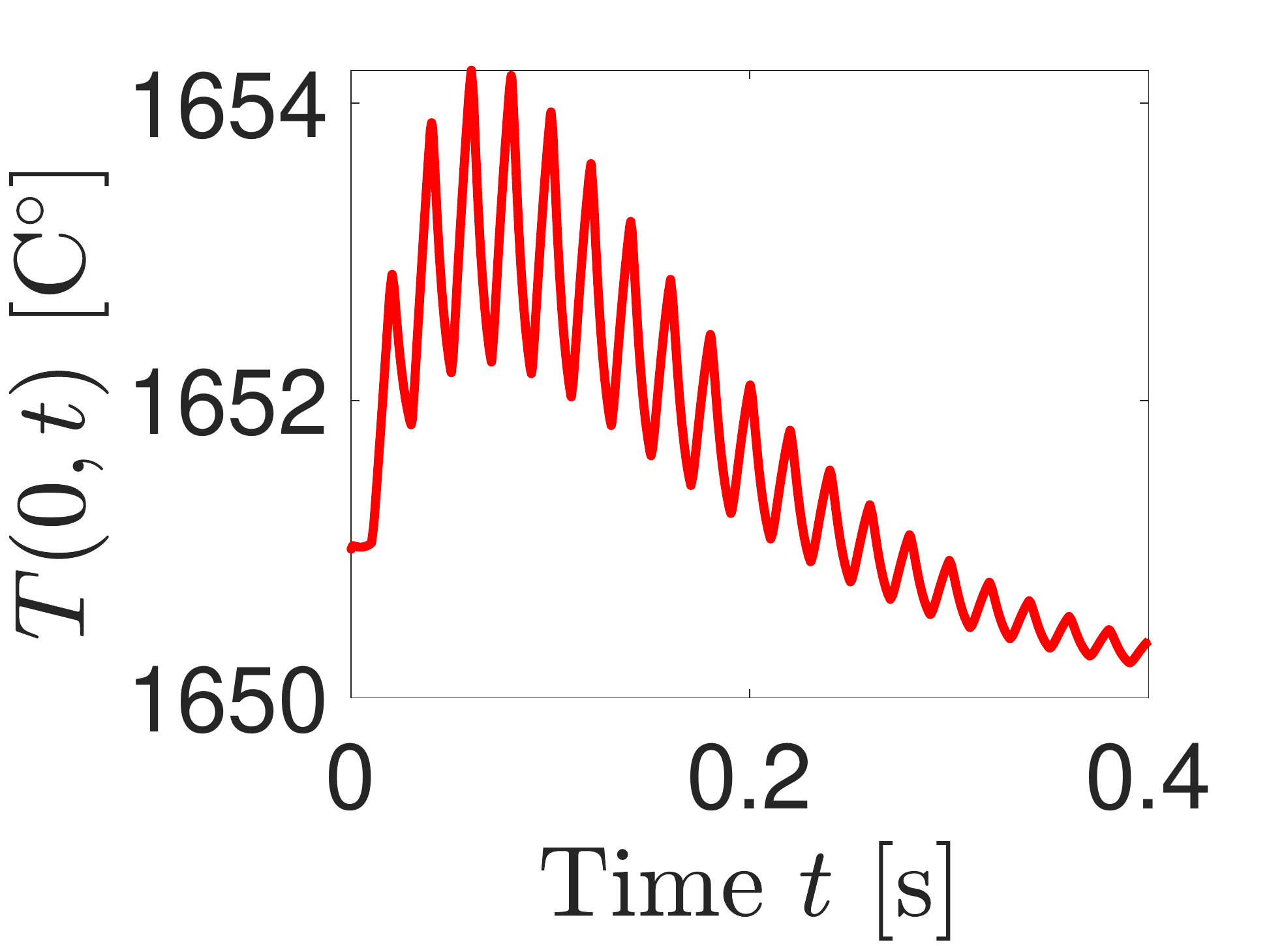}\label{fig:qp_temp}}\hspace{1mm}
\subfloat[Inlet heat flux fluctuates but remains positive---it never cools.]
{\includegraphics[width=2.1in]{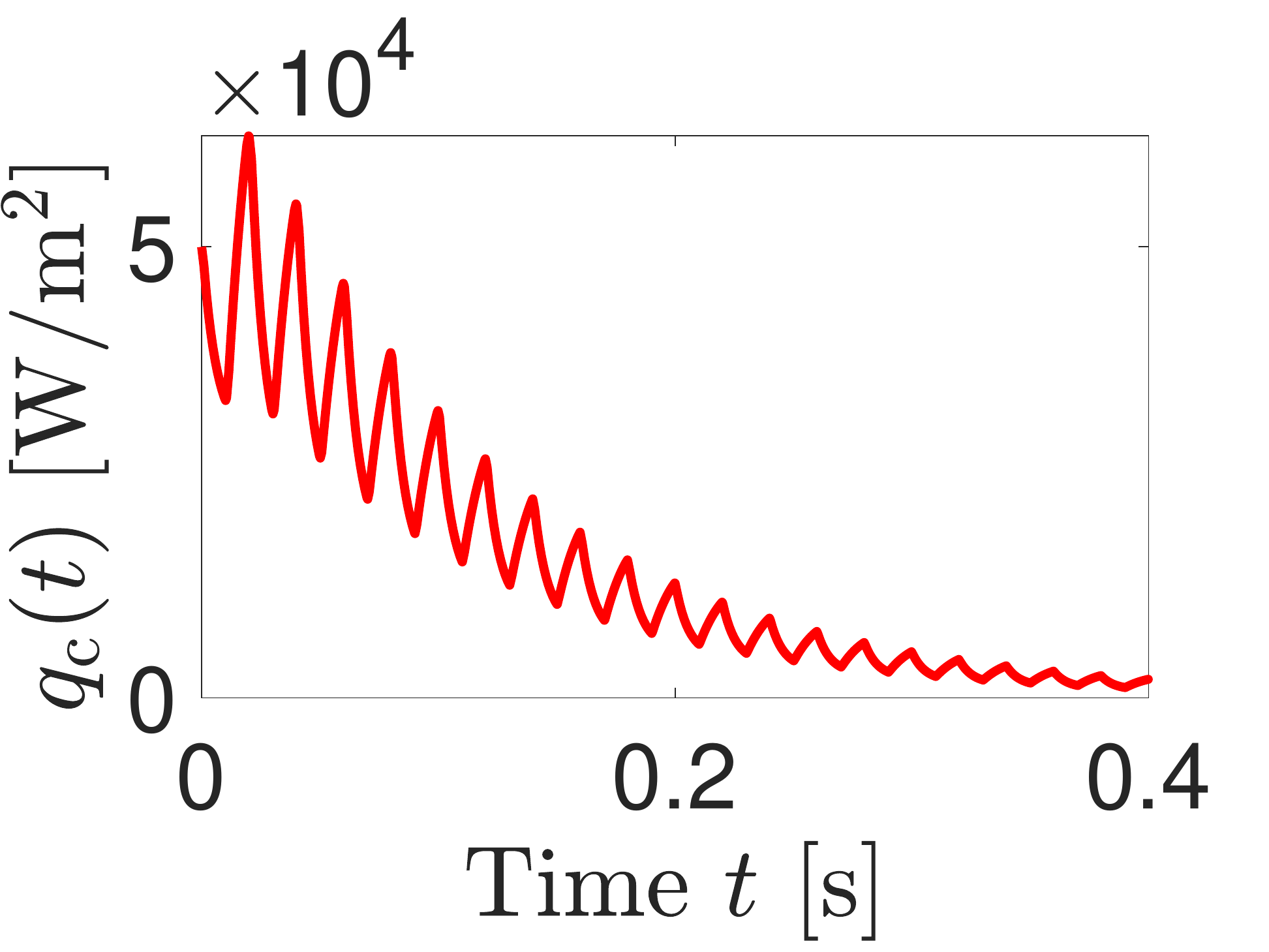}\label{fig:qp_heat}}\\
\subfloat[Energy CBF remains positive.]
{\includegraphics[width=2.1in]{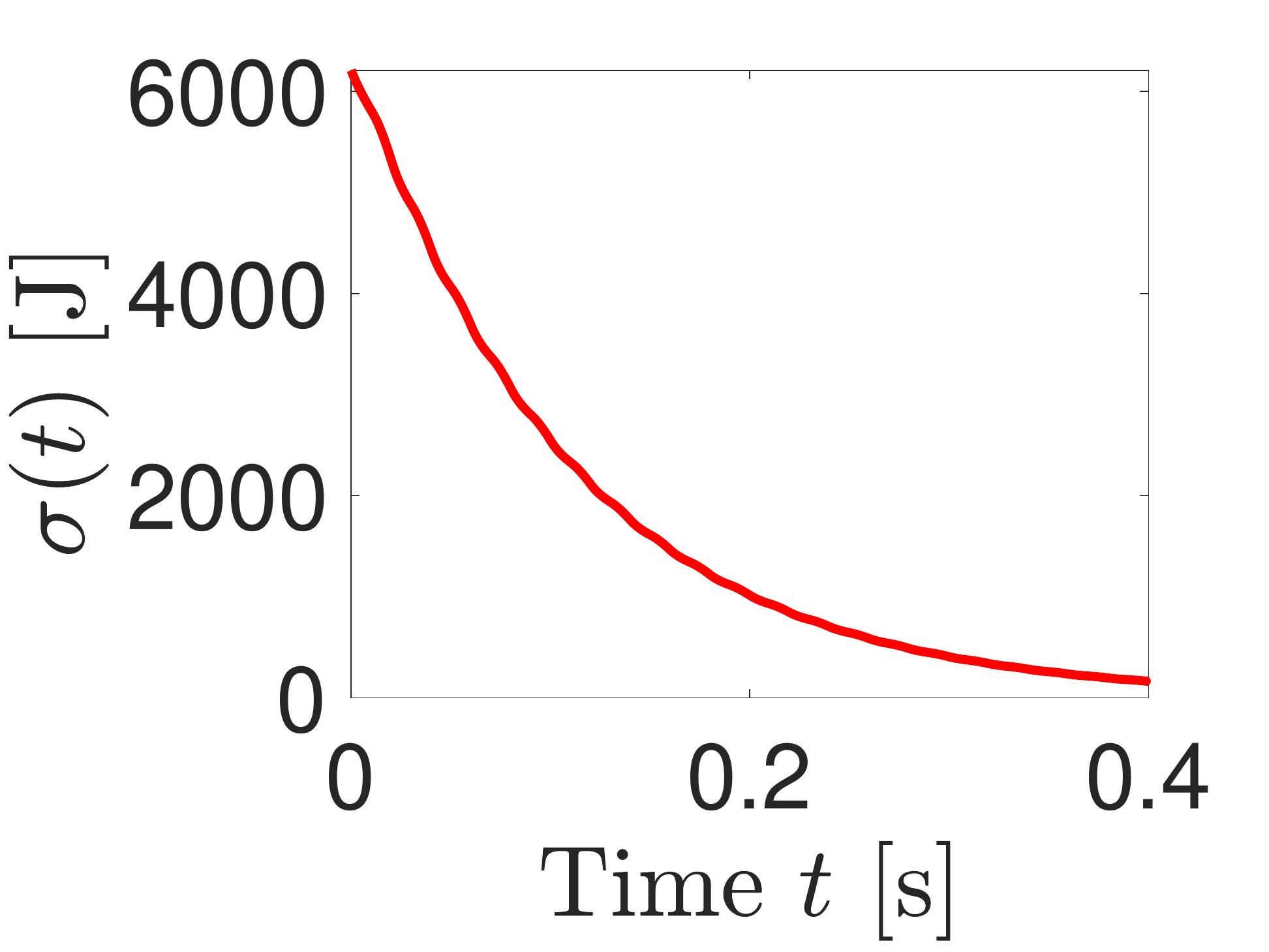}\label{fig:qp_energy}}\hspace{1mm}
\subfloat[QP-backstepping voltage input is kept between the lower and upper bounds. ]
{\includegraphics[width=2.1in]{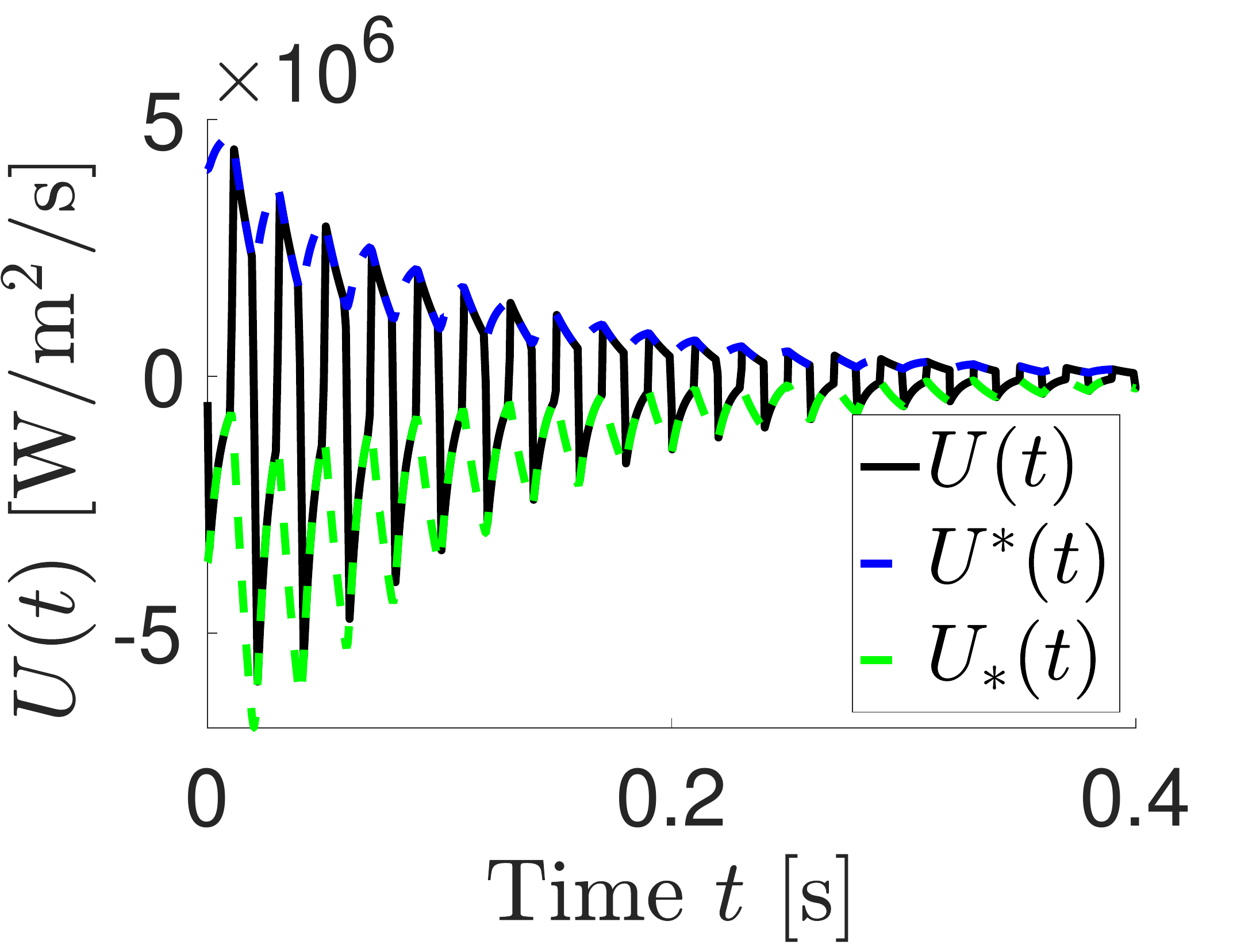}\label{fig:qp_input}}
\caption{Under operator input $U_o(t) = A \sin (\omega t) + B$, which commands both a periodic addition of heat and a net/average removal of heat, and would lead to islands of solid, the QP-backstepping safety input $U$ in plot (e) ensures that the laser $q_{\rm c}$ in plot (c) remains positive, $s$ does not exceed $s_{\rm r}$, and $T$ does not drop below $T_{\rm m}$. The interface $s$ advances because the operator periodically commands addition of heat, in spite of the net command being for a removal of heat ($B<0$).}
\label{fig:qp_result2}
\end{center} 
\end{figure*}

Next, we conduct the simulation of the QP-backstepping safety control \eqref{eq:safety-fin-sat}, with lower bound \eqref{eq:U_star} and upper bound \eqref{eq:U^star}. We set the operator input as $U_o(t) = A \sin (\omega t) + B$, where $A =  1.14 \times 10^7$, $B = -5 \times 10^5$, and $\omega = 2\pi / \tau$ with time period $\tau =  0.02$ [s]. The control gains are set as $k_1 = 64.4$ [/s], $k_2 = 973$ [/s$^2$], $\delta_1 = 129$ [/s], and $\delta_2 = 195 $ [/s$^2$], respectively. Fig. \ref{fig:qp_result2} depicts the result of the closed-loop response. Fig. \ref{fig:qp_result2} (a) illustrates that the interface position monotonically increases with maintaining $s(t)<s_{\rm r}$, and slowly converges to the setpoint position. Fig. \ref{fig:qp_result2} (b) shows QP-backstepping voltage input (black solid), the lower bound (green dash) and the upper bound (blue dash). We observe that the input is affected by both the upper and lower bounds, to maintain the value between them, while it execute the operator input other than that, which is a sinusolidal wave input. We can also see that, as time passes, the lower bound gradually corresponds to the upper bound, as well as the QP input, thereby the regulation of the interface position can be achieved slowly in Fig. \ref{fig:qp_result2} (a). Due to the operator input, Fig. \ref{fig:qp_result2} shows that the boundary temperature is fluctuating through repeating the warming up and cooling down, with maintaining above the melting temperature. A similar behavior can be observed in Fig. \ref{fig:qp_result2} (e) showing the laser power. Also in this setup, as observed in Fig. \ref{fig:result2} (d) and (e), the two CBFs imposed in the problem satisfy the positivity, which ensures the desired performance of the safety filter, being consistent with Theorem \ref{thm:safety}.

\section{Conclusion} \label{sec:conclusion} 

This paper has developed two safe control strategies for the one-phase Stefan PDE-ODE system with actuator dynamics by utilizing CBF. The first one is the non-overshooting control, which is derived by recursive construction of CBFs for ensuring the safety imposed by the physical model, and also achieves the regulation of the moving interface position at a desired setpoint position. The second one is the safety filter design, which employs an operator input under the conditions of maintaining the safety, as a QP-Backstepping-CBF formulation. The stability proof under the non-overshooting control has been achieved by PDE-backstepping method and Lyapunov analysis. The developed methods have been extended to the case of additional constraints from above on the states, the system with higher-order actuator dynamics, and the two-phase Stefan system with an unknown disturbance. The proposed non-overshooting control and the QP safety filter design have been demonstrated in numerical simulation of the process in metal additive manufacturing, which illustrates the desired performance. Namely, both the regulation and safety are simultaneously achieved by the non-overshooting control, while the safety filter affects the system by a chosen operator input with maintaining the safety.



\ifCLASSOPTIONcaptionsoff
  \newpage
\fi



%


\bibliographystyle{plain}
\bibliography{BIB_TAC.bib}

\end{document}